\documentclass[12pt]{amsart}
\usepackage{amssymb,amsmath,amsthm,amscd,mathrsfs}
\usepackage[utf8]{inputenc}
\usepackage{xcolor,graphicx,enumerate}
\usepackage{hyperref,multicol,mdframed}
\usepackage[all]{xy}

\usepackage{newtxtext,newtxmath}

\newcounter{results}
\newtheorem{lemma}[results]{Lemma}
\newtheorem{example}[results]{Example}
\newtheorem{definition}[results]{Definition}
\newtheorem{remark}[results]{Remark}
\newtheorem{theorem}[results]{Theorem}
\newtheorem*{theorema}{Theorem A}
\newtheorem*{theoremb}{Theorem B}
\newtheorem*{theoremc}{Theorem C}
\newtheorem{proposition}[results]{Proposition}
\newtheorem{corollary}[results]{Corollary}

\newcommand{\diffto}{\xrightarrow{\raisebox{-0.2 em}[0pt][0pt]{\smash{\ensuremath{\sim}}}}}
\newcommand{\rmap}{\longrightarrow}

\newcommand{\acts}{\curvearrowright}
\newcommand{\X}{\mathfrak{X}}

\newcommand{\dd}{\mathrm{d}}

\newcommand{\pr}{\mathrm{pr}}

\renewcommand{\dd}{\mathrm{d}}

\newcommand{\id}{\mathrm{id}}

\makeatletter
\def\smallunderbrace#1{\mathop{\vtop{\m@th\ialign{##\crcr
   $\hfil\displaystyle{#1}\hfil$\crcr
   \noalign{\kern3\p@\nointerlineskip}%
   \tiny\upbracefill\crcr\noalign{\kern3\p@}}}}\limits}
\makeatother

\usepackage{geometry}
\usepackage{xparse,pifont}

\newcounter{backups} 

\hypersetup{
        colorlinks=true, 
        linkcolor=blue,  
        citecolor=red, 
        urlcolor=cyan,   
        filecolor=magenta 
    }

\usepackage{cancel}

\title{Homotopy of representations up to homotopy}

\author[O. Brahic]{Olivier Brahic}
\address{Departamento de Matem\'atica, UFPR, Curitiba-PR}
\email{brahicolivier@gmail.com}

\author[P. Frejlich]{Pedro Frejlich}
\address{Departamento de Matem\'atica Pura e Aplicada, UFRGS, Porto Alegre-RS}
\email{frejlich.math@gmail.com}

\newcounter{litreview} 

\newcommand{\lit}{%
  \refstepcounter{litreview}
  \textbf{\Roman{litreview})}
  \ 
}

\begin{document}

\begin{abstract}
We propose a general definition of morphism of representations up to homotopy of Lie algebroids, which is be mandated by the requirement that taking cohomology with coefficients in a representation up to homotopy behave functorially. With this definition in place, we show that such representations behave in a similar way to honest representations as far as homotopy invariance of cohomology and local triviality of smooth families are concerned, and for essentially the same reasons as in the classical case.
\end{abstract}

\dedicatory{To Rui Loja Fernandes, with warm thanks on the occasion of his 60th birthday.}
\maketitle

\tableofcontents

\section{Introduction}
Homotopy invariance of de Rham cohomology is a basic fact of life: the endpoints of a smooth homotopy
\begin{align*}
 h : M \times I \to N
\end{align*}
induce the same maps in de Rham cohomology:
\begin{align*}
 h_0^* = h_1^* : \mathrm{H}(N) \to \mathrm{H}(M).
\end{align*}
There are important variations on this construction, allowing coefficients to be taken in a flat vector bundle -- in fact, a representation of the tangent bundle -- or even in a local system, all of which enjoy the property of homotopy invariance.

The notion of \emph{Lie algebroid} simultaneously generalizes smooth manifolds and Lie
algebras. Lie algebroids turn out to be useful in order to study, in a unified way, foliations, Poisson manifolds, and infinitesimal actions of Lie algebras. A Lie algebroid $A \Rightarrow M$ can be defined in terms of a de Rham-like differential $\mathrm{d}_A : \Omega(A) \to \Omega(A)$ on the space of forms $\Omega(A):=\Gamma(\wedge A^*)$ on $A$, and morphisms are those vector bundle maps $\phi : A \to B$ whose induced map on forms $\phi^* : \Omega(B) \to \Omega(A)$ is a cochain map.

It was already observed in the foundational \cite{CrainicFernandes_annals} that the correct notion of homotopy for Lie algebroids is internal --- namely, a morphism $\phi : A \times TI \to B$, rather than a smooth family of Lie algebroids $A \to B$. For such a homotopy, the maps in cohomology
\begin{align*}
 \phi_0^* = \phi_1^* : \mathrm{H}(B) \to \mathrm{H}(A).
\end{align*}
induced by the endpoints $\phi_0,\phi_1 : A \to B$ coincide \cite{Frejlich,Brahic20,Gengoux_Lavau_Strobl,Jotz_Marchesini}. This can be stated more generally for cohomology with coefficients in a representation $\nabla : B \curvearrowright E$ of $B$, which is likewise described by a cochain structure on $\Omega(B;E):=\Omega(B) \otimes \Gamma(E)$: given a Lie homotopy $\phi : A \times TI \to B$, there is an induced isomorphism of between the pullback representations $\psi : \phi_0^*(E) \diffto \phi_1^*(E)$, for which
\begin{align*}
 \xymatrix{
 & \mathrm{H}(B,E) \ar[dl]_{\phi_0^*} \ar[dr]^{\phi_1^*} & \\
 \mathrm{H}(A,\phi_0^*(E)) \ar[rr]_{\psi}^{\simeq} & & \mathrm{H}(A,\phi_1^*(E))
 }
\end{align*}
commutes. 

Unfortunately, as is long known, the notion of representation of Lie algebroids is too strict. In particular, there is in general no ``adjoint representation'' of a Lie algebroid. This motivated Abad and Crainic \cite{Abad_Crainic} to introduce a more flexible notion of \emph{representation up to homotopy} (or \emph{ruth} for short)
\begin{align*}
 \mathscr{R} \ : \ A \stackrel{\mathrm{D}}{\curvearrowright} \mathcal{E}.
\end{align*}
Concretely, a ruth on $A$ equips the space $\Omega(A;\mathcal{E})$ of forms on $A$ with values in a \emph{graded} vector bundle $\mathcal{E}$ with the structure of a DGA module over the DGA $\Omega(A)$ of forms on $A$. This makes it possible to associate with every Lie algebroid an ``adjoint'' representation up to homotopy. 

For actual representations, there is a semidirect construction which assigns, to a representation $\nabla : A \curvearrowright E$, a Lie algebroid $\mathcal{A}_{\nabla} = A \oplus E$ over $A$. For ruths, no such semidirect construct exists in the category of Lie algebroids, and it would be reasonable to expect that generalizing the results above for representations up to homotopy might entail a fundamentally new argument.

The main goal of the present note is to show that the main gist of the arguments in \cite{Frejlich} carries over to the setting of representations up to homotopy, generalizing the results mentioned above. The actual implementation of the argument turns out to be relatively straightforward as long as one interprets \emph{morphisms} and \emph{flows} appropriately.

\vspace*{0.2cm}
\begin{mdframed}
 A morphism between ruths
 \begin{align*}
  & \mathscr{R}_A \ : \ A \stackrel{\mathrm{D}_{\mathscr{R}_A}}{\curvearrowright} \mathcal{E}
  & \mathscr{R}_B \ : \ B \stackrel{\mathrm{D}_{\mathscr{R}_B}}{\curvearrowright} \mathcal{F}
 \end{align*}
is a morphism of modules of DGAs
 \begin{align*}
  \Phi : \Omega(B;\mathcal{F}^*) \to \Omega(A;\mathcal{E}^*)
 \end{align*}
 over the DGA map given by a morphism of Lie algebroids $\phi : A \to B$.
\end{mdframed}
\vspace*{0.2cm}

The choice of dual coefficients is not merely a matter of convention; it is forced by
functoriality. Just as for ordinary representations, morphisms naturally induce pullback maps on cohomology only after passing to the dual representation.

This definition is in a sense \emph{dual} to the existing definitions of morphism of ruths, and it entails regarding forms with coefficients in the \emph{dual} ruth $\mathcal{E}^*$ as the basic object\footnote{See Remark \ref{rem : dual coefficients} and comments I)-III) in Section \ref{sec : Remarks} for a detailed discussion.}. That is a point of view we espouse and advocate. A strong argument in its favor is that, in this setting, homotopy invariance takes its classical form\footnote{Without dualizing the coefficient ruth, morphisms induce relations instead of maps.}:

\vspace*{0.2cm}
\begin{mdframed}
 \begin{theorema}
 The endpoints of a morphisms of ruths $\Phi : \mathscr{R}_A \times TI \to \mathscr{R}_B$ induce the same map in cohomology:
 \begin{align*}
  \Phi_0^* = \Phi_1^* : \mathrm{H}(\mathscr{R}_B^*) \to \mathrm{H}(\mathscr{R}_A^*).
 \end{align*}
\end{theorema}
\end{mdframed}
\vspace*{0.2cm}

As with usual representations, a morphism of Lie algebroids $\phi : A \to B$ induces a ruth $\phi^{\star}(\mathscr{R})$ on $A$ from a ruth $\mathscr{R}$ on $B$. Our main result is that the pullbacks of a ruth along homotopic Lie algebroid morphisms are \emph{isomorphic}; more precisely:
\vspace*{0.2cm}
\begin{mdframed}
 \begin{theoremb}
 Let $\phi : A \times TI \to B$ be a Lie algebroid homotopy, and let $\mathscr{R}$ be a ruth on $B$. Then $\phi^{\star}(\mathscr{R})$ and $\phi_0^{\star}(\mathscr{R}) \times TI$ are isomorphic ruths on $A \times TI$.
\end{theoremb}
\end{mdframed}
\vspace*{0.2cm}

In fact, there is a \emph{fibered} picture underlying these constructions. Call a \emph{submersion by Lie algebroids} the data of a Lie algebroid $B \Rightarrow N$, and a surjective submersion $p$ whose fibres are transverse to $B$. A choice of ruth $\mathscr{R}$ on $B$ is a \emph{submersion by ruths}. Any smooth map $f:M \to \underline{N}$ induces a pullback \emph{submersion by ruths}
\begin{align*}
 && f^{\bullet}(\mathscr{R}) = (f^{\bullet}(B) \curvearrowright f^*(N) \times_N \mathcal{E}),
 && f^{\bullet}(B) \Rightarrow f^*(N),
 && p: f^*(N) \to M
\end{align*}
and a morphism of submersions by ruths $f_{\bullet} : f^{\bullet}(\mathscr{R}) \to \mathscr{R}$. We think of submersions by ruths as a smooth family of ruths
\begin{align*}
  i_x^{\bullet}(\mathscr{R}) \ \ \text{on} \ \ i_x^{\bullet}(B) \Rightarrow p^{-1}(x),
\end{align*}
parametrized by $x \in \underline{N}$ and assembled in a ``Lie way''. A submersion by ruths is locally trivial if every $x \in \underline{N}$ has an open neighborhood $U \subset \underline{N}$, for which
\begin{align*}
 && i_U^{\bullet}(\mathscr{R})
 && \text{and}
 && i_x^{\bullet}(\mathscr{R}) \times TU
\end{align*}
are isomorphic submersions by ruths over $\mathrm{id}_U$.

\vspace*{0.2cm}
\begin{mdframed}
 \begin{theoremc}
A submersion by ruths is locally trivial exactly when its underlying submersion by Lie algebroids is locally trivial.
\end{theoremc}
\end{mdframed}
\vspace*{0.2cm}

The paper is organized as follows:
\begin{itemize}
	\item \underline{Section \ref{sec : Lie algebroids}} sets up the notation concerning Lie algebroids and their morphisms, emphasizing the differential graded algebra (DGA) point of view. 
	
	\item Coefficients for representations of Lie algebroids appear in \underline{Section \ref{sec : Representations}} in the DGA perspective, and we highlight the covariance issues in the definition of morphism between representations of Lie algebroids over different bases, which plays a crucial role later.
	
	\item  The key notion of homotopy of Lie algebroid morphisms is introduced in \underline{Section \ref{sec : Homotopies}}, where we recall that fibre integration gives a canonical model for a homotopy operator. 
	
	\item The most important instance of a homotopy for this note is the one given by the (local) flow of a section of a Lie algebroid, which we discuss in \underline{Section \ref{sec : Flows in Lie algebroids}}. 
	
	\item The main object of this note, representations up to homotopy (``ruth'' for short), are introduced in \underline{Section \ref{sec : Representations up to homotopy}}, where we propose a general definition of morphism and discuss various notions of pullback ruth.
	
	\item Theorem A is proved in \underline{Section \ref{sec : Homotopy invariance of cohomology}}, which turns out to be straightforward in our set-up. 
	
	\item Theorem B is more delicate, and is explained in \underline{Section \ref{sec : Local triviality}}, where we also address the version of homotopy invariance for the pullback of a ruth under a homotopy. Finally, we highlight in \underline{Section \ref{sec : The fibred picture}} that the discussion in the paper has a relative version, where ruths are defined on the total space of a \emph{fibrations by Lie algebroids} --- that is, extensions of Lie algebroids satisfying a certain completeness condition. Those include submersions by Lie algebroids as a basic example.
\end{itemize} 

\subsection*{Acknowledgements} The authors acknowledge financial support from CNPq Universal grant 402320/2023-9.

\section{Lie algebroids}\label{sec : Lie algebroids}

A \emph{Lie algebroid} $A \Rightarrow M$ is a vector bundle $A$ on $M$, equipped with a bundle map $\rho_A : A \to TM$ covering the identity, and a Lie bracket $[\cdot,\cdot]_A : \Gamma(A) \times \Gamma(A) \to \Gamma(A)$, such that
\begin{align*}
 & [a,fb]_A = f[a,b]_A+(\mathscr{L}_{\rho_A(a)f})b, & f \in C^{\infty}(M), \ a,b \in \Gamma(A).
\end{align*}
The whole structure of $A$ as a Lie algebroid can be dually defined as a linear map

\begin{center}
\begin{minipage}{0.5\textwidth}
 \begin{mdframed}[backgroundcolor=green!5]
 \vspace{-0.2cm}
\begin{align*}
 \dd_A : \Omega(A) \to \Omega(A), 
\end{align*}  
 \end{mdframed}
\end{minipage}
\end{center}
its \emph{de Rham differential}, which:

\vspace{0.2cm}
\begin{minipage}{0.9\textwidth}
 \begin{mdframed}[backgroundcolor=olive!3]
\begin{enumerate}[align-left]
 \item[LAlg1)] has degree one, $\dd_A : \Omega^p(A) \to \Omega^{p+1}(A)$;
 \item[LAlg2)] is a derivation of $\wedge$:
 \begin{align*}
\dd_A(\omega \wedge \eta) = (\dd_A\omega)\wedge \eta+(-1)^{p}\omega \wedge (\dd_A\eta),
\end{align*}
 for all $(\omega,\eta) \in \Omega^p(A) \times \Omega^q(A)$;
 \item[LAlg3)] squares to zero: $\mathrm{d}_{A}^2=0$.
\end{enumerate}  
 \end{mdframed}
\end{minipage}
\vspace{0.2cm}

Said otherwise, $\Omega(A)$ has the structure of a \emph{differential graded algebra}.

If $B \Rightarrow N$ is another Lie algebroid, and
\begin{align*}
 \xymatrix{
 A \ar@{=>}[d] \ar[r]^{\phi} & B \ar@{=>}[d] \\
 M \ar[r]_{\underline{\phi}} & N
 }
\end{align*}
is a vector bundle map, then there is an induced algebra homomorphism
  \begin{align*}
   & \phi^* : \Omega(B) \to \Omega(A), & (\phi^*\omega)_x(a_1,...,a_p):=\omega_{\underline{\phi}(x)}(\phi(a_1),...,\phi(a_p)),
  \end{align*}
and we call $\phi : A \to B$ a \emph{\bf morphism} of Lie algebroids if $\phi^*$ is a cochain map:
\begin{align*}
\dd_A \phi^* = \phi^* \dd_B.
\end{align*}
Lie algebroid homomorphisms $A \to B$ are thus in bijective correspondence with linear maps

\begin{center}
\begin{minipage}{0.5\textwidth}
 \begin{mdframed}[backgroundcolor=magenta!5]
 \vspace{-0.2cm}
\begin{align*}
 \Phi : \Omega(B) \to \Omega(A)
\end{align*}  
 \end{mdframed}
\end{minipage}
\end{center}
which are homomorphisms of differential graded algebras --- that is, which

\vspace{0.2cm}
\begin{minipage}{0.9\textwidth}
 \begin{mdframed}[backgroundcolor=olive!3]
\begin{enumerate}[align-left]
 \item[MLAlg1)] have degree zero;
 \item[MLAlg2)] commute with $\wedge$:
 \begin{align*}
  &\Phi(\omega \wedge \eta) = \Phi(\omega) \wedge \Phi(\eta), & \omega,\eta \in \Omega(B);
 \end{align*}
 \item[MLAlg3)] are a cochain maps:
 \begin{align*}
  &\Phi \mathrm{d}_{B} = \mathrm{d}_{A} \Phi.
 \end{align*}
\end{enumerate} 
 \end{mdframed}
\end{minipage}
\vspace{0.2cm}

 \begin{example}[Pullback of a Lie algebroid by a transverse map]\label{ex : Pullback of a Lie algebroid by a transverse map}\normalfont
 Let $B \Rightarrow N$ be a Lie algebroid, and let $f:M \to N$ be a smooth map which is \emph{transverse} to $B$, in the sense that
 \begin{align}\label{eq : transverse map}
  Tf(T_xM) + \rho_B(B_{f(x)}) = T_{f(x)}N
 \end{align}
 for all $x \in M$. Then the \emph{pullback Lie algebroid} $f^!(B) \Rightarrow M$,
 \begin{align*}
  f^!(B):= TM \times_{TN}B = \{(v,b) \in TM \times B \ | \ Tf(v)=\rho_B(b)\}
 \end{align*}
 is defined, and comes equipped with a \emph{pullback morphism of Lie algebroids} $f_!(v,b) = b$:
 \begin{align*}
  \xymatrix{
   f^!(B) \ar[r] \ar@{=>}[d]
 & B \ar@{=>}[d] \\
 M \ar[r]_f 
 & N
  }
 \end{align*}
\end{example}

\section{Lie algebroid representations}\label{sec : Representations}
A \emph{representation} of $A \Rightarrow M$ on a vector bundle $E$ on $M$ is given by an \emph{$A$-connection}, that is, a linear map
\begin{align*}
 & \nabla : \Gamma(E) \to \Omega(A;E), & e \mapsto \nabla e,
\end{align*}
into the space $\Omega(A;E):=\Omega(A) \otimes_{C^{\infty}(M)}\Gamma(E)$ of forms on $A$ with values in $E$, such that
\begin{align*}
 \nabla (fe) = f\nabla e + \mathrm{d}_Af \otimes e
\end{align*}
for all $(f,e) \in C^{\infty}(M) \times \Gamma(E)$, which is \emph{flat}:
\begin{align*}
 & \nabla_{[a,b]_A} = \left[\nabla_a,\nabla_b\right], & a,b \in \Gamma(A).
\end{align*}
There is then a bijective correspondence\footnote{Where $\mathrm{d}_{\nabla}(\omega)(a_0,....,a_p)$, for $\omega \in \Omega^p(A;E)$ and $a_0,...,a_p \in \Gamma(A)$, is given explicitly by
\begin{align*}
  \sum_{i=0}^p (-1)^i\nabla_{a_i}\omega(a_0,...,a_{i-i},a_{i+1},...,a_p) + \sum_{i<j}(-1)^{i+j}\omega([a_i,a_j],a_0,...,a_{i-i},a_{i+1},...,a_{j-i},a_{j+1},...,a_p).
\end{align*}
} between representations $\nabla : A \curvearrowright E_A$ of $A$ on $E$, and linear maps

\begin{center}
\begin{minipage}{0.5\textwidth}
 \begin{mdframed}[backgroundcolor=green!5]
 \vspace{-0.2cm}
\begin{align*}
 \dd_{\nabla} : \Omega(A;E) \to \Omega(A;E), 
\end{align*}  
 \end{mdframed}
\end{minipage}
\end{center}
which:

\vspace{0.2cm}
\begin{minipage}{0.9\textwidth}
 \begin{mdframed}[backgroundcolor=olive!3]
\begin{enumerate}[align-left]
 \item[Rep1)] have degree one;
 \item[Rep2)] are derivations of $\wedge$:
 \begin{align*}
\dd_{\nabla}(\omega \wedge \eta) = (\dd_A\omega)\wedge \eta+(-1)^{p}\omega \wedge (\dd_{\nabla}\eta),
\end{align*}
 for all $(\omega,\eta) \in \Omega^p(A) \times \Omega^q(A)$;
 \item[Rep3)] square to zero: $\mathrm{d}_{\nabla}^2=0$.
\end{enumerate}
 \end{mdframed}
\end{minipage}
\vspace{0.2cm}

\noindent $\Omega(A;E)$ is thus a ``module'' over the differential graded algebra $\Omega(A)$.

\begin{example}\normalfont
 The \emph{dual} of a representation $\nabla : A \curvearrowright E$ is the representation $\nabla^* : A \curvearrowright E^*$ determined by the condition that
 \begin{align*}
  \mathscr{L}_{\rho_A(a)} \langle \zeta, e\rangle = \langle \nabla^*_a\zeta, e\rangle + \langle \zeta,\nabla_ae\rangle
 \end{align*}
 for all $(a,\zeta,e) \in \Gamma(A) \times \Gamma(E^*)\times \Gamma(E)$. Clearly, $(\nabla^*)^*=\nabla$.
\end{example}

\begin{example}[Semi-direct product]\normalfont
 Given a representation $\nabla : A \curvearrowright E$, there is a Lie algebroid structure on $\mathbf{R}_A = A \oplus E$, with anchor and bracket given by
 \begin{align*}
  & \rho_{\mathbf{R}_A} = \rho_A \circ \mathrm{pr}_A,
  & \left[ \begin{pmatrix}
            a_1 \\ e_1
           \end{pmatrix},
           \begin{pmatrix}
            a_2 \\ e_2
           \end{pmatrix}
 \right] = \begin{pmatrix}
            [a_1,a_2]_A \\ \nabla_{a_1}e_2 - \nabla_{a_2}e_1
           \end{pmatrix}.
 \end{align*}
Note that $\mathrm{pr}_A:\mathbf{R}_A \to A$ is a surjective Lie algebroid map over $\mathrm{id}_M$ which admits a splitting $\sigma : A \to \mathbf{R}_A$ which is a Lie algebroid map. 
\end{example}

	Fix representations $\nabla^A : A \curvearrowright E_A$ and $\nabla^B : B \curvearrowright E_B$ of $A \Rightarrow M$ and $B \Rightarrow N$ respectively, together with a Lie algebroid map and a vector bundle map as follows
	\begin{align*}
		\begin{gathered}
			\xymatrix{
				A \ar@{=>}[d] \ar[r]^{\phi} & B \ar@{=>}[d] \\
				M \ar[r]_{\underline{\phi}} & N,
			}
		\end{gathered}&&
		\begin{gathered}
			\xymatrix{
				E_A \ar[d] \ar[r]^{\psi} & E_B \ar[d] \\
				M \ar[r]_{\underline{\phi}} & N.
		}\end{gathered}
	\end{align*}	
Together, $\phi$ and $\psi$ define a map
$\Phi : \Omega(B,E_B^*) \to \Omega(A,E_A^*)$ completely determined by:
\begin{align*}
	& \langle \Phi(\eta)_x(a_1,...,a_p),e \rangle:= \langle \eta_{\underline{\phi}(x)}(\phi(a_1),...,\phi(a_p)),\psi(e) \rangle
\end{align*}
for all $x \in M$, $a_1,...,a_p \in A_x$ and $e \in E_{A,x}$. The following are then equivalent:
\begin{enumerate}[i)]
	\item $\Phi : \Omega(B,E_B^*) \to \Omega(A,E_A^*)$ is a cochain map;
	\item $(\phi,\psi):\mathbf{R}_A \to \mathbf{R}_B$ is a Lie algebroid map,
\end{enumerate}
\noindent in which case we call the pair $(\phi,\psi)$ a \emph{morphism of representations}. Algebraically, such morphisms translate into linear maps
\begin{center}
\begin{minipage}{0.5\textwidth}
 \begin{mdframed}[backgroundcolor=magenta!5]
 \vspace{-0.2cm}
\begin{align*}
 \Phi : \Omega(B;E_B^*) \to \Omega(A;E_A^*)
\end{align*}  
 \end{mdframed}
\end{minipage}
\end{center}
which are homomorphisms of modules over differential graded algebras, meaning that:

\vspace{0.2cm}
\begin{minipage}{0.9\textwidth}
 \begin{mdframed}[backgroundcolor=olive!3]
\begin{enumerate}[align-left]
 \item[MRep1)] $\Phi$ has degree zero;
 \item[MRep2)] $\Phi$ is a module homomorphism over $\phi^*$:
 \begin{align*}
  &\Phi(\omega \wedge \eta) = \phi^*(\omega) \wedge \Phi(\eta), & (\omega,\eta) \in \Omega(B) \times \Omega(B;E_B^*);
 \end{align*}
 \item[MRep3)] $\Phi$ is a cochain map:
 \begin{align*}
  &\Phi \mathrm{d}_{{\nabla^B}^*} = \mathrm{d}_{{\nabla^A}^*} \Phi.
 \end{align*}
\end{enumerate}
 \end{mdframed}
\end{minipage}
\vspace{0.2cm}

\begin{remark}\label{rem : dual coefficients}\normalfont
Taking cohomology of representations of Lie algebroids only defines a \emph{functor} if said cohomology is taken with coefficients in the \emph{dual} representation. That the dualization of coefficients is required in order that a morphism of representations give rise to a morphism in cohomology was first observed in \cite[Section 1.2]{Schwarzbach_Gengoux_Weinstein}. Without dualizing coefficients, a morphism $(\phi,\psi):\mathbf{R}_A \to \mathbf{R}_B$ of representations of Lie algebroids only induces a \emph{relation}
\begin{align*}
 \mathrm{Rel} \subset \Omega(A,E_A) \times \Omega(B,E_B)
\end{align*}
as follows: factor $\psi : E_A \to E_B$ as the composition of
\begin{align*}
 & \overline{\psi} : E_A \to \underline{\phi}^*(E_B),
 & \phi : \underline{\phi}^*(E_B) \to E_B.
\end{align*}
Then $(\eta_A,\eta_B) \in \Omega(A,E_A) \times \Omega(B,E_B)$ lies in $\mathrm{Rel}$ iff
\begin{align*}
 \overline{\psi}(\eta_A) = \phi^*(\eta_B) \in \Omega(A,\underline{\phi}^*(E_B)).
\end{align*}
There are two special instances in which a morphism of representations $(\phi,\psi):\mathbf{R}_A \to \mathbf{R}_B$ does induce a cochain map $\Omega(B,E_B) \to \Omega(A,E_A)$, namely:
\begin{itemize}
 \item when the vector bundle map $\psi : E_A \to E_B$ is fibrewise invertible, or
 \item when $A$ and $B$ are Lie algebroids on the same manifold $M$, and $\underline{\phi}=\mathrm{id}_M$.
\end{itemize}
\end{remark}

\section{Homotopies}\label{sec : Homotopies}

Let $I=[0,1]$ stand for the unit interval, and $TI \Rightarrow I$ for its tangent bundle. 
A {\bf homotopy} between Lie algebroids $A \Rightarrow M$ and $B \Rightarrow N$ is a Lie algebroid map $\phi : A \times TI \to B$. Two Lie algebroid maps	$\phi_0 : A \to B$ and  $\phi_1 : A \to B$ are {\bf homotopic} if a homotopy $\phi : A \times TI \to B$ exists, with $\phi_0 = \phi \circ i_0$ and $\phi_1 = \phi \circ i_1$, where
\begin{align}\label{eq : endpoints of A times TI}
	& i_0 : A \to A \times TI, & i_1 : A \to A \times TI
\end{align}
denote the \emph{endpoints} of $A \times TI$, that is, the canonical identifications of $A$ with $A \times T\{0\}$ and $A \times T\{1\}$:
\begin{align*}
	\xymatrix{
		& B & \\
		A \ar[ur]^{\phi_0} \ar[r]_{i_0 \phantom{123}} & A \times TI \ar[u]_{\phi} & A \ar[l]^{\phantom{123}i_1} \ar[ul]_{\phi_1}
	}
\end{align*}

\vspace{0.2cm}
\begin{minipage}{0.9\textwidth}
 \begin{mdframed}[backgroundcolor=blue!5]
\begin{proposition}\label{pro : homotopy op lie alg}
 A Lie algebroid homotopy $\phi : A \times TI \to B$ determines a \emph{canonical} cochain homotopy
 \begin{align*}
  \mathfrak{h}_{\phi} : \Omega^{\bullet}(B) \to \Omega^{\bullet-1}(A)
 \end{align*}
 between the cochain maps $\phi_1^*,\phi_0^*:\Omega^{\bullet}(B) \to \Omega^{\bullet}(A)$ induced by the endpoints of $\phi$:
 \begin{align*}
  \phi_1^*(\omega)-\phi_0^*(\omega) = \dd_A \mathfrak{h}_{\phi}(\omega)+\mathfrak{h}_{\phi}\left(\dd_B\omega \right).
 \end{align*}
\end{proposition}
\end{mdframed}
\end{minipage}
\begin{proof}
Indeed, consider the linear map of {\bf fibre integration} \cite{Balcerzak,CrainicFernandes}: 
\begin{align}\label{eq : fibre integral lie algebroid}
 & \smallfint : \Omega^{p+1}(A \times TI) \to \Omega^p(A), & \left(\smallfint \omega \right)(a_1,...,a_p):=\smallint_0^1 \omega(\tfrac{\partial}{\partial t},a_1,...,a_p) \dd t.
\end{align}
Then $\smallfint$ defines a chain homotopy between $i_0^*$ and $i_1^*$: for all $\omega \in \Omega(A \times TI)$,
\begin{align*}
 \dd_A (\smallfint \omega) + \smallfint ( \dd_{A \times TI}\omega) = i_1^*(\omega)-i_0^*(\omega),
\end{align*}
where $i_0,i_1 : A \to A \times TI$ denote the endpoints \eqref{eq : endpoints of A times TI}. One can then define an operator
\begin{align*}
 & \mathfrak{h}_{\phi} : \Omega^{\bullet}(B) \to \Omega^{\bullet-1}(A), & \mathfrak{h}_{\phi}(\omega):= \smallfint \phi^*(\omega),
\end{align*}
and check that:
\begin{align*}
 \dd_A \mathfrak{h}_{\phi}(\omega)+\mathfrak{h}_{\phi}(\dd_B \omega) & = \dd_A\smallfint \phi^*(\omega) +\smallfint \phi^*(\dd_B \omega)\\
 & = \dd_A\smallfint \phi^*(\omega) +\smallfint \dd_{A \times TI}\phi^*(\omega)\\
 & = i_1^*\phi^*(\omega) - i_0^*\phi^*(\omega)\\
 & = \phi_1^*(\omega) - \phi_0^*(\omega).
 \end{align*} 
\end{proof}

 \vspace{0.2cm}
\begin{minipage}{0.9\textwidth}
 \begin{mdframed}[backgroundcolor=blue!5]
 \begin{corollary}
  Homotopic Lie algebroid maps induce the same map in cohomology:
 \begin{align*}
  \phi_0^* = \phi_1^* : \mathrm{H}^{\bullet}(B) \to \mathrm{H}^{\bullet}(A).
 \end{align*}
 \end{corollary}
\end{mdframed}
\end{minipage}
 \vspace{0.2cm}

\section{Flows in Lie algebroids}\label{sec : Flows in Lie algebroids}

Every vector bundle $p:E \to M$ gives rise to a Lie algebroid $\mathfrak{D}(E) \Rightarrow M$ of \emph{derivative operators} \cite{Schwarzbach_Mackenzie}. Its space of sections $\Gamma(\mathfrak{D}(E))$ can be identified with those $\mathbb{R}$-linear maps $L : \Gamma(E) \to \Gamma(E)$, for which there is a vector field $v \in \X(M)$, such that
\begin{align*}
 & L(fs) = (\mathscr{L}_vf)s+f(Ls), & (f,s) \in C^{\infty}(M) \times \Gamma(E).
\end{align*}
Every section $s \in \Gamma(E)$ gives rise to a vertical vector field on $E$,
\begin{align*}
 & s^{\mathrm{v}} \in \X_{\mathrm{vert}}(E), & s^{\mathrm{v}}(e):=\tfrac{d}{d\epsilon}\left(e+\epsilon s(p(e)) \right)|_{\epsilon=0} \in T_eE.
\end{align*}
Derivative operators $L:\Gamma(E) \to \Gamma(E)$ are in bijective correspondence with linear vector fields $\mathcal{V} \in \mathfrak{X}_{\mathrm{lin}}(E)$, satisfying the condition that
\begin{align*}
 [\mathcal{V},s^{\mathrm{v}}] = (Ls)^{\mathrm{v}}
\end{align*}
for all $s \in \Gamma(E)$. The vector field $\mathcal{V}$ is related, via the canonical projection $p:E \to M$, to a vector field $v \in \mathfrak{X}(M)$ and therefore $p$ intertwines their flows:
\begin{align*}
 p \circ \phi_{\epsilon}^{\mathcal{V}} = \phi_{\epsilon}^v \circ p
\end{align*}
wherever they are defined. If $v$ is a complete vector field, then $\mathcal{V}$ is also a complete vector field, in which case their flows define a smooth family of vector bundle isomorphisms
\begin{align*}
 \xymatrix{
 E \ar[d] \ar[r]^{\phi_{\epsilon}^{\mathcal{V}}} & E \ar[d] \\
 M \ar[r]_{\phi_{\epsilon}^v} & M
 }
\end{align*}
Denote by $\phi^{\dagger}_{\epsilon} : \Gamma(E) \to \Gamma(E)$ the map
\begin{align*}
& \phi^{\dagger}_{\epsilon}(e)(x) := \left(\phi_{\epsilon}^{\mathcal{V}}\right)^{-1}\left(e(\phi_{\epsilon}^v(x))\right), 
& (x,e) \in M \times \Gamma(E).
\end{align*}
Then the following confition holds:
\begin{align}\label{eq : flow of derivative operator}
 & \tfrac{d}{d\epsilon} \phi_{\epsilon}^{\dagger}(s) = \phi_{\epsilon}^{\dagger}(Ls), & s \in \Gamma(E),
\end{align}
and it characterizes the (local) flow of $L$.

\begin{example}[Flow of a section of a Lie algebroid]\label{ex : Flow of a section of a Lie algebroid}\normalfont
 Let $A \Rightarrow M$ be a Lie algebroid, and $a \in \Gamma(A)$ be a section whose anchor $v=\rho_A(a) \in \X(M)$ is a complete vector field. Then
 \begin{align*}
  L = [a,\cdot] : \Gamma(A) \to \Gamma(A)
 \end{align*}
is a derivative operator, and its flow $\phi_{\epsilon}$ is by Lie algebroid isomorphisms:
\begin{align*}
 & \phi_{\epsilon}^{\dagger}[b,c] = [\phi_{\epsilon}^{\dagger}(b),\phi_{\epsilon}^{\dagger}(c)], 
 & b,c \in \Gamma(A).
\end{align*}
Thus
\begin{align*}
 &\phi : A \times I \to A, & \phi(b,\epsilon) = \phi_{\epsilon}(b)
\end{align*}
is a Lie algebroid morphism, the {\bf flow of $a$} \cite[Appendix A]{CrainicFernandes_annals}. Consider the unique vector bundle map
\begin{align*}
 \phi + a \mathrm{d}\epsilon : A \times TI \to A
\end{align*}
which restricts to $\phi$ on $A \times I$, and which maps $\tfrac{\partial}{\partial \epsilon}$ to $a$. Then $\phi + a \mathrm{d}\epsilon$ is a Lie algebroid morphism. Indeed, since
\begin{align*}
 \left(\phi + a \mathrm{d}\epsilon\right)^*(\omega) = \phi_{\epsilon}^*(\omega) + \mathrm{d}\epsilon \wedge \phi_{\epsilon}^*(i_a\omega),
\end{align*}
we obtain, as a consequence:
\begin{multline*}
 \mathrm{d}_{A \times TI}\left(\phi + a \mathrm{d}\epsilon\right)^*(\omega) - \left(\phi + a \mathrm{d}\epsilon\right)^*(\mathrm{d}_B\omega) = \\ = 
 \left( \mathrm{d}_A\phi_{\epsilon}^*(\omega) - \phi_{\epsilon}^*\mathrm{d}_B\omega \right) + \mathrm{d}\epsilon \wedge \left(\tfrac{d}{d\epsilon}\phi_{\epsilon}^*(\omega) - \phi_{\epsilon}^*(\mathscr{L}_a\omega)  \right) = 0.
\end{multline*}
\end{example}
Thus:

\vspace{0.2cm}
\begin{minipage}{0.9\textwidth}
 \begin{mdframed}[backgroundcolor=blue!5]
The flow of a complete section $a \in \Gamma(A)$ defines a Lie algebroid homotopy
 \begin{align*}
  \phi + a \mathrm{d}\epsilon : A \times TI \to A.
 \end{align*}  
 \end{mdframed}
\end{minipage}

\section{Representations up to homotopy}\label{sec : Representations up to homotopy}

Let $A \Rightarrow M$ be a Lie algebroid, and
\begin{align*}
 \mathcal{E} = \oplus_j \mathcal{E}_j
\end{align*}
be a graded vector bundle on $M$. Equip the space $\Omega(A;\mathcal{E}) =\Omega(A)\otimes \Gamma(\mathcal{E})$ with the total grading
\begin{align*}
  \Omega^p(A;\mathcal{E}) = \oplus_{i+j=p} \Omega^i(A;\mathcal{E}_j).
 \end{align*}

\begin{definition}[Ruth]\label{def : ruth}\normalfont
A {\bf representation up to homotopy} $\mathscr{R}$ of $A$ on $\mathcal{E}$ consists of a linear map

\begin{center}
\begin{minipage}{0.5\textwidth}
 \begin{mdframed}[backgroundcolor=green!5]
 \vspace{-0.2cm}
\begin{align*}
 \mathrm{D}_{\mathscr{R}} : \Omega(A;\mathcal{E}) \to \Omega(A;\mathcal{E}),
\end{align*}  
 \end{mdframed}
\end{minipage}
\end{center}
which:

\vspace{0.2cm}
\begin{minipage}{0.9\textwidth}
 \begin{mdframed}[backgroundcolor=olive!3]
\begin{enumerate}[align-left]
 \item[Ruth1)] \label{Ruth1} has degree one, 
 \begin{align*}
  \mathrm{D}_{\mathscr{R}}\Omega^p(A;\mathcal{E}) \subset \Omega^{p+1}(A;\mathcal{E});
 \end{align*}
 \item[Ruth2)] \label{Ruth2} is a derivation: for all $\omega \in \Omega^p(A)$ and $\eta \in \Omega^q(A;\mathcal{E})$,
 \begin{align*}
  \mathrm{D}_{\mathscr{R}}(\omega \wedge \eta) = \mathrm{d}_A(\omega) \wedge \eta + (-1)^p \omega \wedge \mathrm{D}_{\mathscr{R}}(\eta);
 \end{align*}
 \item[Ruth3)] \label{Ruth3} squares to zero,
 \begin{align*}
  \mathrm{D}_{\mathscr{R}}^2=0.
 \end{align*}
\end{enumerate}
 \end{mdframed}
\end{minipage}
\end{definition} 

For brevity, we will refer to a representation up to homotopy as a {\bf ruth}, and write
\begin{align*}
 \mathscr{R} \ : \ A \overset{\mathrm{D}_{\mathscr{R}}}{\curvearrowright} \mathcal{E}
\end{align*}
as a shorthand notation. We think of the ruth $\mathscr{R}$ as an object ``over $A$''. We denote by
\begin{align*}
 & \Omega(\mathscr{R}):=\Omega(A,\mathcal{E}), & \mathrm{H}(\mathscr{R}):=\mathrm{H}\left(\Omega(A,\mathcal{E}),\mathrm{D}_{\mathscr{R}}\right)
\end{align*}
the ensuing modules of forms and cohomology (respectively over the rings $\Omega(A)$ and $\mathrm{H}(A)$).

\begin{example}[Ruths of length 1]\normalfont
 If $\mathcal{E}$ is concentrated in degrees 0 and 1, a ruth $\mathscr{R} \ : \ A \overset{}{\curvearrowright} \mathcal{E}$ consists of
 \begin{enumerate}[align-left]
  \item[a)] a vector bundle map $\partial : \mathcal{E}_0 \to \mathcal{E}_1$;
  \item[b)] $A$-connections $\nabla^{\mathcal{E}_i} : A \curvearrowright \mathcal{E}_i$ which are compatible with $\partial$,
  \begin{align*}
   \partial \nabla^{\mathcal{E}_0} = \nabla^{\mathcal{E}_1} \partial;
  \end{align*}
 \item[c)] a two-form $K \in \Omega^2(A;\mathcal{E}_1^* \otimes \mathcal{E}_0)$ satisfying
 \begin{align*}
  && \mathrm{R}_{\nabla^{\mathcal{E}_0}} = K \circ \partial,
  && \mathrm{R}_{\nabla^{\mathcal{E}_1}} = \partial \circ K,
  && \mathrm{d}_{\nabla}(K)=0,
 \end{align*}
 where $\nabla : A \curvearrowright \mathcal{E}_1^* \otimes \mathcal{E}_0$ is the connection induced by $\nabla^{\mathcal{E}_0}$ and $\nabla^{\mathcal{E}_1}$.
 \end{enumerate}
\end{example}

\begin{example}[Dual ruth]\normalfont
Let $\mathcal{E}$ be a graded vector bundle, with dual
\begin{align*}
 & \mathcal{E}^*:=\bigoplus_j (\mathcal{E}^*)_j, & (\mathcal{E}^*)_j = (\mathcal{E}_{-j})^*.
\end{align*}
Then the bilinear pairing
 \begin{align*}
  & \langle\cdot\,,\cdot\rangle:\mathcal{E}^* \times_M \mathcal{E} \to \mathbb{R}
 \end{align*}
induces a degree-zero linear map
 \begin{align}\label{eq : pairing as deg zero}
\langle\cdot\,,\cdot\rangle : \mathcal{E}^* \otimes \mathcal{E} \to \underline{\mathbb{R}},
 \end{align}
 where $\underline{\mathbb{R}}$ denotes the graded vector space consisting of $M \times \mathbb{R}$ in degree zero. Let now $\mathscr{R} = (\mathrm{D}_{\mathscr{R}} : A \acts \mathcal{E})$ be a ruth. Then \eqref{eq : pairing as deg zero} induces a pairing \cite[Example 4.1]{Abad_Crainic}
 \begin{align*}
  \langle\cdot,\cdot\rangle : \Omega(A,\mathcal{E}^*) \times \Omega(A,\mathcal{E}) \to \Omega(A)
 \end{align*}
given, for all $\eta \in \Omega^p(A;(\mathcal{E})^*_{i})$ and $\tau \in \Omega^q(A;\mathcal{E}_j)$, by
 \begin{align*}
  \langle \eta, \tau \rangle (a_1,...,a_{p+q}) & = \sum_{\sigma \in \mathfrak{S}_{(p,q)}} (-1)^{qi}\mathrm{sgn}(\sigma) \langle \eta(a_{\sigma(1)},...,a_{\sigma(p)}),\tau(a_{\sigma(p+1)},...,a_{\sigma(p+q)}) \rangle,
 \end{align*}
where $\mathfrak{S}_{(p,q)}$ denotes the set of $(p,q)$-shuffles. Note that the expression above vanishes if $i+j \neq 0$. 

The {\bf dual} ruth $\mathscr{R}^* = (\mathrm{D}_{\mathscr{R}^*} : A \acts \mathcal{E}^*)$ is uniquely determined by the condition that
 \begin{align*}
  \dd_A \langle \eta, \tau \rangle = \langle \mathrm{D}_{\mathscr{R}^*}(\eta), \tau \rangle + (-1)^{p} \langle \eta ,\mathrm{D}_{\mathscr{R}}(\tau)\rangle.
 \end{align*}
\end{example}

\subsection*{Morphisms of ruths}

Fix ruths
\begin{align*}
 & \mathscr{R}_A \ : \ A \overset{\mathrm{D}_{\mathscr{R}_A}}{\curvearrowright} \mathcal{E}_A,
 & \mathscr{R}_B \ : \ B \overset{\mathrm{D}_{\mathscr{R}_B}}{\curvearrowright} \mathcal{E}_B,
\end{align*}
on the respective Lie algebroids $A \Rightarrow M$ and $B \Rightarrow N$.

\vspace*{0.2cm}
\begin{mdframed}[backgroundcolor=gray!5]
\begin{definition}[Morphism of ruths]\label{def : Morphism of ruths}\normalfont
A {\bf morphism of ruths} $\Phi : \mathscr{R}_A \to \mathscr{R}_B$ consists of a Lie algebroid morphism $\phi : A \to B$, together with a linear map
\begin{center}
\begin{minipage}{0.5\textwidth}
 \begin{mdframed}[backgroundcolor=magenta!5]
 \vspace{-0.2cm}
\begin{align*}
 \Phi : \Omega(\mathscr{R}_B^*) \to \Omega(\mathscr{R}_A^*)
\end{align*}  
 \end{mdframed}
\end{minipage}
\end{center}
which:

\vspace{0.2cm}
\begin{minipage}{0.9\textwidth}
 \begin{mdframed}[backgroundcolor=olive!3]
\begin{enumerate}[align-left]
 \item[MRuth1)] has degree zero;
 \item[MRuth2)] is a module homomorphism over $\phi$;
 \begin{align*}
  &\Phi(\omega \wedge \eta) = \phi^*(\omega) \wedge \Phi(\eta), & (\omega,\eta) \in \Omega(B) \times \Omega(B;\mathcal{E}_B^*);
 \end{align*}
 \item[MRuth3)] is a cochain map:
 \begin{align*}
  &\Phi \circ \mathrm{D}_{\mathscr{R}_B^*} = \mathrm{D}_{\mathscr{R}_A^*} \circ \Phi.
 \end{align*}
\end{enumerate}
 \end{mdframed}
\end{minipage}
\vspace{0.2cm} 
\end{definition} 
\end{mdframed}

\begin{example}[Adjoint ruth]\label{ex : adjoint rep}\normalfont
 Let $\mathrm{ad}(A)$ denote the graded vector bundle with
 \begin{align*}
  & \mathrm{ad}(A)_0 = A,
  & \mathrm{ad}(A)_1 = TM.
 \end{align*}
A choice of linear connection $\nabla : TM \curvearrowright A$ induces $A$-connections
\begin{align*}
 & \nabla^{\mathrm{bas}} : A \curvearrowright A,
 & \nabla^{\mathrm{bas}}_ab:=\nabla_{\rho_A(b)}a+[a,b]_A,\\
 & \nabla^{\mathrm{bas}} : A \curvearrowright TM,
 & \nabla^{\mathrm{bas}}_av:=\rho_A(\nabla_{v}a)+[\rho_A(a),v].
\end{align*}
Denote also by $\mathrm{R}^{\mathrm{bas}}_{\nabla} \in \Omega^2(A;T^*M \otimes A)$ the two-form determined by
\begin{align*}
 \mathrm{R}^{\mathrm{bas}}_{\nabla}(a,b)v:=\nabla_v[a,b]_A-[\nabla_v(a),b]_A-[a,\nabla_vb]_A-\nabla_{\nabla^{\mathrm{bas}}_vb}a+\nabla_{\nabla^{\mathrm{bas}}_va}b
\end{align*}
for all $a,b \in \Gamma(A)$ and $v \in \Gamma(TM)$. Then
\begin{align*}
 \mathrm{D} = \rho_A + \mathrm{d}_{\nabla^{\mathrm{bas}}} + \mathrm{R}^{\mathrm{bas}}_{\nabla}
\end{align*}
is the structure map for a ruth $\mathscr{R} \ : \ A \curvearrowright \mathrm{ad}(A)$. If $\nabla' : TM \curvearrowright A$ is another choice of linear connection, and $\mathscr{R}' \ : \ A \overset{\mathrm{D}'}{\curvearrowright}\mathrm{ad}(A)$ is the ensuing ruth, there is an isomorphism of ruths \cite[Theorem 3.11]{Abad_Crainic}
\begin{align*}
 \Phi : \mathscr{R} \diffto \mathscr{R}'
\end{align*}
whose components are
\begin{align*}
 & \Phi^0 = \mathrm{id},
 & \Phi^1 = \nabla-\nabla'.
\end{align*}
The isomorphism class of this ruth is called the {\bf adjoint ruth} of $A$.
\end{example}

\subsection{Pullbacks of ruths by a Lie algebroid morphism}\label{subsec : Pullbacks of ruths by a Lie algebroid morphism}
 Let $\mathscr{R}_B = (\mathrm{D}_{\mathscr{R}_B} : B \acts \mathcal{E})$ be a ruth. Given a Lie algebroid morphism
 \begin{align*}
  \xymatrix{
  A \ar@{=>}[d] \ar[r]^{\phi} & B \ar@{=>}[d]\\
  M \ar[r]_{\underline{\phi}} & N
  }
 \end{align*}
write $\underline{\phi}^*(\mathcal{E})= M \times_N\mathcal{E}$ for the pullback bundle, and
 \begin{align}\label{eq : pullback morphism of ruths}
  \phi^* : \Omega(B;\mathcal{E}^*) \to  \Omega(A;\underline{\phi}^*(\mathcal{E}^*))
 \end{align}
for the induced map of modules. Because $\Omega(A;\underline{\phi}^*(\mathcal{E}^*))$ is spanned, as an $\Omega(A)$-module, by sections $\phi^*(\zeta) \in \Gamma(\underline{\phi}^*(\mathcal{E}^*))$, where $\zeta \in \Gamma(\mathcal{E}^*)$, a unique ruth
 \begin{align}\label{eq : pullback ruth}
  \phi^{\star}(\mathscr{R}_B) = (\mathrm{D}_{\phi^{\star}(\mathscr{R}_B)} : A \acts \underline{\phi}^*\mathcal{E}),
 \end{align}
exists, for which
\begin{align*}
 & \mathrm{D}_{\phi^{\star}(\mathscr{R}_B)} \phi^*(\zeta) = \phi^*(\mathrm{D}_{\mathscr{R}_B}\zeta), & \zeta \in \Gamma(\mathcal{E}^*).
\end{align*}
We call \eqref{eq : pullback ruth} the {\bf pullback ruth} of $\mathscr{R}_B$ by $\phi$. Note that \eqref{eq : pullback morphism of ruths} gives a {\bf pullback morphism} of ruths
\begin{align*}
 \phi_{\star} : \phi^{\star}(\mathscr{R}_B) \to \mathscr{R}_B.
\end{align*}

\vspace{0.2cm}
\begin{minipage}{0.9\textwidth}
 \begin{mdframed}[backgroundcolor=blue!5]
\begin{lemma}\label{lem : morphisms factor through pullback}
 Any morphism of ruths
 \begin{align*}
  \Phi : \mathscr{R}_A \to \mathscr{R}_B
 \end{align*}
covering a Lie algebroid morphism $\phi : A \to B$ factors through the pullback:
 \begin{align*}
  \xymatrix{
  & \phi^{\star}(\mathscr{R}_B) \ar[dr]^{\phi_{\star}} &\\
  \mathscr{R}_A \ar[ur]^{\overline{\Phi}} \ar[rr]_{\Phi} & & \mathscr{R}_B
  }
 \end{align*}
\end{lemma}
 \end{mdframed}
\end{minipage}
\begin{proof}
First observe that \eqref{eq : pullback morphism of ruths} preserves the bigrading
\begin{align*}
 \phi^*\Omega^p(B;(\mathscr{E}_B)_i) \subset \Omega^p(\phi^!(B);(\underline{\phi}^*\mathscr{E}_B)_i)
\end{align*}
and that $\phi^*\Gamma(\mathscr{R}_B^*)$ spans $\Omega(\phi^{\star}(\mathscr{R}_B^*))$ over $\Omega(A)$. Therefore a unique morphism of modules of DGAs
\begin{align*}
 \overline{\Phi} : \Omega(\phi^{\star}(\mathscr{R}_B^*)) \to \Omega(\mathscr{R}_A^*)
\end{align*}
exists, for which
\begin{align*}
 \overline{\Phi}(\phi^*(\eta)):=\Phi(\eta)
\end{align*}
for each $\eta \in \Omega(\mathscr{R}_B^*)$. Therefore
\begin{align*}
 \overline{\Phi}(\mathrm{D}_{\phi^{\star}(\mathscr{R}_B^*)}\phi^*(\eta)) = \overline{\Phi}(\phi^*(\mathrm{D}_{\mathscr{R}_B^*}\eta)) = \Phi(\mathrm{D}_{\mathscr{R}_B^*}\eta)) = \mathrm{D}_{\mathscr{R}_A^*}\Phi(\eta) = \mathrm{D}_{\mathscr{R}_A^*}\overline{\Phi}(\phi^*(\eta))
\end{align*}
implies that $\overline{\Phi}$ is a morphism of ruths.
\end{proof}

\subsection{Pullbacks of ruths by a transverse map}\label{subsec : Pullbacks of ruths by a transverse map}

 Let $B \Rightarrow N$ be a Lie algebroid, let $f:M \to N$ be a smooth map which is \emph{transverse} to $B$, and let $f_! : f^!(B) \to B$ be the pullback morphism of Lie algebroids of Example \ref{ex : Pullback of a Lie algebroid by a transverse map}. If $\mathscr{R}_B = (\mathrm{D}_{\mathscr{R}_B} : B \acts \mathcal{E})$ is a ruth on $B$, we write
 \begin{align*}
  (f_!)^{\star}(\mathscr{R}_B) \ : \ f^!(B) \overset{\mathrm{D}_{(f_!)^{\star}(\mathscr{R}_B)}}{\curvearrowright} f^*(\mathcal{E})
 \end{align*}
 to denote the pullback of $\mathscr{R}_B$ under the Lie algebroid morphism $f_! : f^!(B) \to B$ as in subsection \ref{subsec : Pullbacks of ruths by a Lie algebroid morphism}, and by
 \begin{align*}
  & (f_!)_{\star} : (f_!)^{\star}(\mathscr{R}_B) \to \mathscr{R}_B
 \end{align*}
 the pullback morphism of ruths.

\vspace{0.2cm}
\begin{minipage}{0.9\textwidth}
 \begin{mdframed}[backgroundcolor=blue!5]
\begin{lemma}\label{lem : pullback morphism factors}
Let $\mathscr{R} = (\mathrm{D}_{\mathscr{R}} : B \acts \mathcal{E})$ be a ruth on $B \Rightarrow N$, and
\begin{align*}
 \xymatrix{
 A \ar[r]^{\phi} \ar@{=>}[d] & B \ar@{=>}[d]\\
 M \ar[r]_{f} & N
 }
\end{align*}
be a morphism of Lie algebroids, for which $f$ is transverse to $B$. Then the morphism of ruths $\phi_{\star} : \phi^{\star}(\mathscr{R}) \to \mathscr{R}$ factors through $(f_!)_{\star} : (f_!)^{\star}(\mathscr{R}) \to \mathscr{R}$:
\begin{align*}
  \xymatrix{
  & (f_!)^{\star}(\mathscr{R}) \ar[dr]^{(f_!)_{\star}} &\\
  \phi^{\star}(\mathscr{R}) \ar[ur]^{\overline{\phi_{\star}}} \ar[rr]_{\phi_{\star}} & & \mathscr{R}
  }
 \end{align*}
\end{lemma}
 \end{mdframed}
\end{minipage}
\begin{proof}
 Since $\phi : A \to B$ factors as the composition of morphisms of Lie algebroids
 \begin{align*}
  \xymatrix{
  A \ar[r]^{\overline{\phi}} \ar@{=>}[d] & f^!(B) \ar[r]^{f_!} \ar@{=>}[d] & B \ar@{=>}[d]\\
  M \ar@{=}[r] & M \ar[r]_{f} & N
  }
 \end{align*}
we see that
 \begin{align*}
  \phi^{\star}(\mathscr{R}) & = \left(f_! \circ \overline{\phi} \right)^{\star}(\mathscr{R})\\
  & = \overline{\phi}^{\star}(f_!)^{\star}(\mathscr{R})\\
  & = \overline{\phi}^{\star}((f_!)^{\star}(\mathscr{R})),
 \end{align*}
 which shows that $\phi_{\star} : \phi^{\star}(\mathscr{R}) \to \mathscr{R}$ is the composition of
 \begin{align*}
  \phi^{\star}(\mathscr{R}) \stackrel{\overline{\phi}_{\star}}{\rmap} (f_!)^{\star}(\mathscr{R}) \stackrel{(f_!)_{\star}}{\rmap} \mathscr{R}.
 \end{align*}
\end{proof}

\section{Homotopy invariance of cohomology}\label{sec : Homotopy invariance of cohomology}

Consider the Lie algebroid map
\begin{align*}
 \xymatrix{
 A \times TI \ar@{=>}[d] \ar[r]^{\mathrm{pr}} & A \ar@{=>}[d] \\
 M \times I \ar[r]_{\mathrm{pr}} & M
 }
\end{align*}
and let
\begin{align*}
 \mathscr{R} \ : \ A \overset{\mathrm{D}_{\mathscr{R}}}{\curvearrowright} \mathcal{E}
\end{align*}
be a ruth on $A$. Consider the pullback ruth
\begin{align*}
 \pr^{\star}(\mathscr{R}) = \mathscr{R} \times TI \ : \ A \times TI \overset{\mathrm{D}_{\pr^{\star}(\mathscr{R})}}{\curvearrowright} \mathcal{E} \times I,
\end{align*}
and observe that there is a unique $C^{\infty}(M)$-linear map
 \begin{align*}
 \smallfint : \Omega^{\bullet}(\mathscr{R}^* \times TI) \to \Omega^{\bullet-1}(\mathscr{R}^*)
\end{align*}
 with the property that
 \begin{align*}
  & \smallfint(\omega \otimes \pr^*(\zeta)) = \left( \smallfint \omega \right) \otimes \zeta
 \end{align*}
for all $\omega \in \Omega^{p}(A \times TI)$ and $\zeta \in  \Gamma((\mathcal{E}^*)_j)$, where $\smallfint \omega$ stands for the fibre integral of \eqref{eq : fibre integral lie algebroid}.

\vspace{0.2cm}
\begin{minipage}{0.9\textwidth}
 \begin{mdframed}[backgroundcolor=blue!5]
\begin{lemma}\label{lem : stokes}
 For all $\eta \in \Omega(\mathscr{R}^* \times TI)$, we have that
\begin{align*}
 \mathrm{D}_{\mathscr{R}^*}\left(\smallfint \eta\right) + \smallfint \mathrm{D}_{\mathscr{R}^* \times TI}(\eta) = i_1^*(\eta)-i_0^*(\eta).
\end{align*}
\end{lemma}
 \end{mdframed}
\end{minipage}

\begin{proof}
 Consider $\eta=\omega \otimes \pr^*(\zeta)$, where $\omega \in \Omega^p(A \times TI)$ and $\zeta \in \Gamma((\mathcal{E}^*)_j)$. By definition of fibre integration,
 \begin{align*}
  \smallfint (\omega \otimes \pr^*(\zeta)) & = \left(\smallfint \omega \right) \otimes \zeta,
 \end{align*}
and therefore
\begin{align*}
 \mathrm{D}_{\mathscr{R}^*} \smallfint (\omega \otimes \pr^*(\zeta)) & = \left(\dd_A \smallfint \omega \right) \otimes \zeta + (-1)^{p-1}\left(\smallfint \omega \right) \wedge \mathrm{D}_{\mathscr{R}^*}\zeta.
\end{align*}
On the other hand,
 \begin{align*}
  \mathrm{D}_{\mathscr{R}^* \times TI}(\omega \otimes \pr^*(\zeta)) & = \dd_{A \times TI}(\omega) \otimes \pr^*(\zeta) + (-1)^p \omega \wedge \mathrm{D}_{\mathscr{R}^* \times TI}\pr^*(\zeta)\\
  & = \dd_{A \times TI}(\omega) \otimes \pr^*(\zeta) + (-1)^p \omega \wedge \pr^*(\mathrm{D}_{\mathscr{R}^*}\zeta)
 \end{align*}
 and so
 \begin{align*}
  \smallfint \mathrm{D}_{\mathscr{R}^* \times TI}(\omega \otimes \pr^*(\zeta)) & = \left(\smallfint \dd_{A \times TI}(\omega) \right)\otimes \zeta + (-1)^p \left(\smallfint \omega \right)\wedge \mathrm{D}_{\mathscr{R}^*}(\zeta)
 \end{align*}
Therefore by \eqref{eq : fibre integral lie algebroid},
 \begin{align*}
  \mathrm{D}_{\mathscr{R}^*} \smallfint \eta + \smallfint \mathrm{D}_{\mathscr{R}^* \times TI}(\eta) & = \left(\smallfint \dd_{A \times TI}(\omega) + \dd_A \smallfint \omega \right)\otimes \zeta\\
  & = \left(i_1^*(\omega) - i_0^*(\omega)\right)\otimes \zeta\\
  & = i_1^*(\eta) - i_0^*(\eta),
 \end{align*}
 and this finishes the proof.
\end{proof}

%

\vspace{0.2cm}
\begin{minipage}{0.9\textwidth}
 \begin{mdframed}[backgroundcolor=red!5]
\begin{theorem}\label{thm : cohomology}
 The endpoints of a homotopy of morphisms of ruths
 \begin{align*}
  \Phi : \mathscr{R}_A \times TI \to \mathscr{R}_B
 \end{align*}
induce the same map in cohomology:
 \begin{align*}
  \Phi_0^* = \Phi_1^* : \mathrm{H}(\mathscr{R}_B^*) \to \mathrm{H}(\mathscr{R}_A^*).
 \end{align*}
\end{theorem}
 \end{mdframed}
\end{minipage}
\begin{proof}
 Let $\Phi : \mathscr{R}_A \times TI \to \mathscr{R}_B$ be a morphism of ruths, and obsrve that the diagram below commutes:
 \begin{align*}
  \xymatrix{
  & \Omega(\mathscr{R}_B^*) \ar[d]^{\Phi} \ar[dl]_{\Phi_0} \ar[dr]^{\Phi_1} & \\
  \Omega(\mathscr{R}_A^*) & \Omega(\mathscr{R}_A^* \times TI) \ar[l]^{i_0^*} \ar[r]_{i_1^*} & \Omega(\mathscr{R}_A^*)
  }
 \end{align*}
Let $\eta \in \Omega(\mathscr{R}_B^*)$, and apply Lemma \ref{lem : stokes} to $\Phi(\eta)$ to deduce that
\begin{align*}
 \Phi_1(\eta)-\Phi_0(\eta) = \mathrm{D}_{\mathscr{R}_A^*}\mathfrak{H}_{\Phi}(\eta) + \mathfrak{H}_{\Phi}(\mathrm{D}_{\mathscr{R}_B^*}\eta),
\end{align*}
where $\mathfrak{H}_{\Phi}$ denotes the homotopy operator
\begin{align*}
 & \mathfrak{H}_{\Phi} : \Omega(\mathscr{R}_B^*) \to \Omega(\mathscr{R}_A^*),
 & \mathfrak{H}_{\Phi}(\eta):=\smallfint \Phi(\eta).
\end{align*}
In particular, if $\mathrm{D}_{\mathscr{R}_B^*}(\eta) = 0$,
\begin{align*}
 \Phi_1(\eta) = \Phi_0(\eta) + \mathrm{D}_{\mathscr{R}_A^*}\mathfrak{H}_{\Phi}(\eta)
\end{align*}
shows that the maps
\begin{align*}
 \Phi_0,\Phi_1 : \mathrm{H}(\mathscr{R}_B^*) \to \mathrm{H}(\mathscr{R}_A^*)
\end{align*}
coincide.
\end{proof}

\section{Local triviality}\label{sec : Local triviality}

A {\bf homotopy} of ruths is a morphism of the form
\begin{align*}
 \Phi : \mathscr{R}_A \times TI \to \mathscr{R}_B,
\end{align*}
where
\begin{align*}
& \mathscr{R}_A \ : \ A \overset{\mathrm{D}_{\mathscr{R}_A}}{\curvearrowright} \mathcal{E}_A,
& \mathscr{R}_B \ : \ B \overset{\mathrm{D}_{\mathscr{R}_B}}{\curvearrowright} \mathcal{E}_B,
\end{align*}
are ruths on $A \Rightarrow M$ and $B \Rightarrow N$, respectively, and
\begin{align*}
 \mathscr{R}_A \times TI = \mathrm{pr}^{\star}(\mathscr{R}_A)
\end{align*}
denotes the pullback ruth of $\mathscr{R}_A$ by the canonical projection $\mathrm{pr}:A \times TI \to A$.

\vspace{0.2cm}
\begin{minipage}{0.9\textwidth}
 \begin{mdframed}[backgroundcolor=blue!5]
\begin{proposition}
There is a canonical construction which, to every complete section $a \in \Gamma(A)$ and ruth $\mathscr{R} = (D_{\mathscr{R}} : A \acts \mathcal{E})$ assigns a homotopy of ruths
\begin{align*}
 \Phi: \mathscr{R} \times TI \to \mathscr{R}
\end{align*}
with the property that
\begin{enumerate}[align-left]
 \item[a)] $\Phi_0 = \id_{\mathscr{R}}$,
 \item[b)] $\Phi_{\epsilon}$ is an isomorphism of ruths for each $\epsilon \in [0,1]$;
 \item[c)] $\tfrac{d}{d\epsilon}\Phi_{\epsilon}^*(\eta) = \Phi_{\epsilon}^*( \mathrm{L}_a \eta)$,
\end{enumerate}
for all $\eta \in \Omega(\mathscr{R}^*)$, where $\mathrm{L}_a$ denotes the degree zero derivation of $\Omega(\mathscr{R}^*)$
\begin{align*}
 \mathrm{L}_a  = [i_a,\mathrm{D}_{\mathscr{R}^*}].
\end{align*}   
\end{proposition}
 \end{mdframed}
\end{minipage}

\begin{proof}
 Consider the vector bundles
 \begin{align*}
  \mathcal{A}^p:= \oplus_{i+j=p}\wedge^iA^* \otimes (\mathcal{E}^*)_j
 \end{align*}
 on $M$, and set $\mathcal{A}:=\oplus_{p \geqslant 0}\mathcal{A}^p$. We thus have identifications
 \begin{align*}
  & \Omega^p(\mathscr{R}^*) = \Gamma(\mathcal{A}^p),
  & \Omega(\mathscr{R}^*) = \Gamma(\mathcal{A}).
 \end{align*}
Let $\mathrm{D}_{\mathscr{R}^*}$ be the structure map of $\mathscr{R}^*$, and let
\begin{align*}
 & i_a : \Omega(\mathscr{R}^*) \to \Omega(\mathscr{R}^*), & (i_a\eta)(a_1,...,a_p):=\eta(a,a_1,...,a_p)
\end{align*}
denote the degree $-1$ derivation of interior product with $a \in \Gamma(A)$. Then
\begin{align*}
 \mathrm{L}_a:=[i_a,\mathrm{D}_{\mathscr{R}^*}]
\end{align*}
is a degree $0$ derivation, and
\begin{align*}
 \mathrm{D}_{\mathscr{R}^*}\mathrm{L}_a = \mathrm{L}_a\mathrm{D}_{\mathscr{R}^*}.
\end{align*}
Because
\begin{align*}
 & \mathrm{L}_a(f\eta) = f(\mathrm{L}_a\eta)+(\mathscr{L}_af)\eta, & (f,\eta) \in C^{\infty}(M) \times \Omega(\mathscr{R}^*),
\end{align*}
$\mathrm{L}$ is a derivative operator $\mathrm{L}:\Gamma(\mathscr{A}) \to \Gamma(\mathscr{A})$ which is compatible with the decomposition $\mathscr{A}=\oplus_p \mathcal{A}^p$. Therefore the linear vector field
\begin{align*}
 \mathcal{V} \in \mathfrak{X}_{\mathrm{lin}}(\mathscr{A})
\end{align*}
that corresponds to $\mathrm{L}_a$ is related to $v:=\rho_A(a) \in \mathfrak{X}(M)$ by the canonical projection $p:\mathscr{A} \to M$, and is tangent to each $\mathscr{A}^p$. Because $\mathcal{V}$ is linear, it is complete provided that $v$ is complete, in which case its flow is by vector bundle isomorphisms
\begin{align*}
 \xymatrix{
 \mathscr{A} \ar[d] \ar[r]^{\phi^{\mathcal{V}}_{\epsilon}}_{\simeq} & \mathscr{A} \ar[d] \\
 M \ar[r]_{\phi^{v}_{\epsilon}}^{\simeq} & M
 }
\end{align*}
This defines maps of sections
\begin{align*}
 & \Phi_{\epsilon} : \Gamma(\mathscr{A}) \to \Gamma(\mathscr{A}),
 & \Phi_{\epsilon}(\eta):=\left(\phi^{\mathcal{V}}_{\epsilon}\right)^{-1} \circ \eta \circ \phi^{v}_{\epsilon},
\end{align*}
or, as we will prefer to write, $\Phi_{\epsilon} : \Omega(\mathscr{R}^*) \to \Omega(\mathscr{R}^*)$. By construction \eqref{eq : flow of derivative operator}, a) and c) are satisfied:
\begin{align*}
 & \Phi_0 = \mathrm{id},
 & \tfrac{d}{d\epsilon}\Phi_{\epsilon}^*(\eta) = \Phi_{\epsilon}^*( \mathrm{L}_a \eta).
\end{align*}
We claim that also b) is satisfied --- that is, that each each $\Phi_{\epsilon}$ is an isomorphism of ruths, covering the Lie algebroid isomorphisms
\begin{align*}
 \varphi_{\epsilon} : A \diffto A
\end{align*}
given by the flow of $a \in \Gamma(A)$ as in Example \ref{ex : Flow of a section of a Lie algebroid}. Let us first observe that, for $\eta \in \Omega(\mathscr{R}^*)$, we have that
\begin{align*}
 [\mathcal{V},\Phi_{\epsilon}(\eta)^{\mathrm{v}}] & = \left(\mathrm{L}_a\Phi_{\epsilon}(\eta)\right)^{\mathrm{v}}
\end{align*}
by definition of the correspondence between $\mathrm{L}_a$ and $\mathcal{V}$, and also
\begin{align*}
 [\mathcal{V},\left(\Phi_{\epsilon}(\eta)\right)^{\mathrm{v}}] & = \left[\mathcal{V},(\phi^{\mathcal{V}}_{\epsilon})^*(\eta^{\mathrm{v}}) \right] = (\phi^{\mathcal{V}}_{\epsilon})^*\left[\mathcal{V},\eta^{\mathrm{v}} \right] = (\phi^{\mathcal{V}}_{\epsilon})^*(\mathrm{L}_a\eta)^{\mathrm{v}}
 = \left(\Phi_{\epsilon}(\mathrm{L}_a\eta)\right)^{\mathrm{v}}.
\end{align*}
Therefore $\Phi_{\epsilon}$ and $\mathrm{L}_a$ commute:
\begin{align}\label{eq : Phi L commute}
 \Phi_{\epsilon} \mathrm{L}_a = \mathrm{L}_a \Phi_{\epsilon}.
\end{align}
\noindent \underline{\emph{$\Phi_{\epsilon}$ satisfies MRuth2).}} Let $\omega \in \Omega(A)$ and $\eta \in \Omega(\mathscr{R}^*)$. Consider the curves
\begin{align*}
 & c_1(\epsilon):=\Phi_{\epsilon}(\omega \wedge \eta),
 & c_2(\epsilon):=\varphi_{\epsilon}^*(\omega) \wedge \Phi_{\epsilon}(\eta)
\end{align*}
in $\Omega(\mathscr{R}^*)$. Because they start at $\omega \wedge \eta$ and because of \eqref{eq : Phi L commute},
\begin{align*}
 \tfrac{d}{d\epsilon}c_1(\epsilon) & = \Phi_{\epsilon}\mathrm{L}_a(\omega \wedge \eta)\\
 & = \mathrm{L}_ac_1(\epsilon),\\
 \tfrac{d}{d\epsilon}c_2(\epsilon) & = \varphi_{\epsilon}^*(\mathscr{L}_a\omega) \wedge \Phi_{\epsilon}(\eta)+\phi_{\epsilon}^*(\omega) \wedge \Phi_{\epsilon}(\mathrm{L}_a\eta) \\
 & = \left(\mathscr{L}_a\varphi_{\epsilon}^*(\omega)\right) \wedge \Phi_{\epsilon}(\eta)+\phi_{\epsilon}^*(\omega) \wedge \left(\mathrm{L}_a\Phi_{\epsilon}(\eta)\right) \\
 & = \mathrm{L}_ac_2(\epsilon)
\end{align*}
we deduce that $c_1,c_2:[0,1] \to \Omega(\mathscr{R}^*)$ satisfy the same initial-value problem
\begin{align*}
 & \tfrac{d}{d\epsilon}c(\epsilon) = \mathrm{L}c(\epsilon),
 & c(0) = \omega \wedge \eta,
\end{align*}
and therefore
\begin{align*}
\Phi_{\epsilon}(\omega \wedge \eta) = \varphi_{\epsilon}^*(\omega) \wedge \Phi_{\epsilon}(\eta).
\end{align*}
\noindent \underline{\emph{$\Phi_{\epsilon}$ satisfies MRuth3).}} Let $\eta \in \Omega(\mathscr{R}^*)$. Consider the curves
\begin{align*}
 & c_1(\epsilon):=\Phi_{\epsilon}(\mathrm{D}_{\mathscr{R}^*}\eta),
 & c_2(\epsilon):=\mathrm{D}_{\mathscr{R}^*}\Phi_{\epsilon}(\eta)
\end{align*}
in $\Omega(\mathscr{R}^*)$, and compute
\begin{align*}
 \tfrac{d}{d\epsilon}c_1(\epsilon) & = \Phi_{\epsilon}\mathrm{L}(\mathrm{D}_{\mathscr{R}^*}\eta) = \mathrm{L}c_1(\epsilon),\\
 \tfrac{d}{d\epsilon}c_2(\epsilon) & = \mathrm{D}_{\mathscr{R}^*}\Phi_{\epsilon}(\mathrm{L}\eta) = \mathrm{L}c_2(\epsilon)
\end{align*}
to deduce that $\Phi_{\epsilon}(\mathrm{D}_{\mathscr{R}^*}\eta) = \mathrm{D}_{\mathscr{R}^*}\Phi_{\epsilon}(\eta)$. \\

\noindent \underline{\emph{$\Phi_{\epsilon}$ assemble into a homotopy of ruths $\mathscr{R} \times TI \to \mathscr{R}$.}} We claim finally that the morphisms of ruths $\Phi_{\epsilon}$ assemble into a homotopy of ruths $\Phi : \mathscr{R} \times TI \to \mathscr{R}$ covering the homotopy $\varphi : A \times TI \to A$ given by the flow of $a$. Indeed, define
\begin{align*}
 & \widetilde{\Phi} : \Omega(\mathscr{R}^*) \to \Omega(\mathscr{R}^* \times TI),
 & \widetilde{\Phi}(\eta) := \Phi_{\epsilon}(\eta) + \mathrm{d}\epsilon \wedge \Phi_{\epsilon}(i_a\eta).
\end{align*}
Then $\widetilde{\Phi}$ satisfies MRuth1) and MRuth2), and because
\begin{align*}
 \mathrm{D}_{\mathscr{R}^* \times TI}\widetilde{\Phi}(\eta) & = \mathrm{D}_{\mathscr{R}^*}\Phi_{\epsilon}^*(\eta)  + \mathrm{d}\epsilon \wedge \tfrac{d}{d\epsilon}\Phi_{\epsilon}^*(\eta) - \mathrm{d}\epsilon \wedge \mathrm{D}_{\mathscr{R}^*}\Phi_{\epsilon}^*(i_a\eta)\\
 \Phi^*(\mathrm{D}_{\mathscr{R}^*}\eta) & = \Phi_{\epsilon}^*(\mathrm{D}_{\mathscr{R}^*}\eta)+ \mathrm{d}\epsilon \wedge \Phi_{\epsilon}^*(i_a\mathrm{D}_{\mathscr{R}^*}\eta)
\end{align*}
and
\begin{align*}
 \tfrac{d}{d\epsilon}\Phi_{\epsilon}^*(\eta) & = \Phi_{\epsilon}^*(\mathrm{L}\eta) = \Phi_{\epsilon}^*(\mathrm{D}_{\mathscr{R}^*}i_a\eta)+\Phi_{\epsilon}^*(i_a\mathrm{D}_{\mathscr{R}^*}\eta),
\end{align*}
we deduce that also MRuth3) holds. This concludes the proof.
\end{proof}

\vspace{0.2cm}
\begin{minipage}{0.9\textwidth}
 \begin{mdframed}[backgroundcolor=red!5]
\begin{theorem}\label{thm : local triviality ruths}
 Let $\phi : A \times TI \to B$ be a morphism of Lie algebroids, and let $\mathscr{R}$ be a ruth on $B$. Then $\phi^{\star}(\mathscr{R})$ and $\phi_0^{\star}(\mathscr{R}) \times TI$ are isomorphic ruths on $A \times TI$.
\end{theorem}
 \end{mdframed}
\end{minipage}
\begin{proof}
Fix a Lie algebroid morphism
\begin{align*}
 \xymatrix{
 A \times TI \ar@{=>}[d] \ar[r]^{\phantom{12}\phi} & B \ar@{=>}[d] \\
 M \times I \ar[r]_{\phantom{12}f} & N
 }
\end{align*}
and a ruth $\mathscr{R}_B \ : \ B \overset{\mathrm{D}_{\mathscr{R}}}{\acts} \mathcal{E}$ on $B$. Consider the pullback ruth
 \begin{align*}
  \phi^{\star}(\mathscr{R}_B) \ : \ A \times TI \overset{\mathrm{D}_{\phi^{\star}(\mathscr{R}_B)}}{\acts} f^*(\mathcal{E})
 \end{align*}
on $A \times TI$. \\

\noindent \underline{\emph{Step 1. The local flow by translations.}} Consider the section
\begin{align*}
 \tfrac{\partial}{\partial \epsilon} \in \Gamma(A \times TI),
\end{align*}
with ensuing degree-zero derivation
\begin{align*}
 \mathrm{L} : \Omega(\phi^{\star}(\mathscr{R}_B^*)) \to \Omega(\phi^{\star}(\mathscr{R}_B^*))
\end{align*}
and linear vector field
\begin{align*}
 & \mathcal{V} \in \mathfrak{X}_{\mathrm{lin}}(\mathcal{A}_{TI}), & \mathcal{A}_{TI} = \bigoplus_{i,j}\wedge^i(A \times TI)^* \otimes f^*(\mathcal{E}^*)_j
\end{align*}
The local flow $\phi^{\mathcal{V}}$ of $\mathcal{V}$ gives a vector bundle isomorphism
\begin{align*}
\xymatrix{
 i_0^*(\mathcal{A}_{TI}) \times [0,1] \ar[d] \ar[r]^{\phantom{1234567}\phi^{\mathcal{V}}}_{\phantom{1234567}\simeq} & \mathcal{A}_{TI} \ar[d] \\
 M \times I \ar@{=}[r] & M \times I,
 }
\end{align*}
where
\begin{align*}
 i_0^*(\mathcal{A}_{TI}) = M \times_{M \times I}\mathcal{A}_{TI}.
\end{align*}
Observe that there is a canonical surjective linear map
\begin{align*}
 q : \Gamma\left( i_0^*(\mathcal{A}_{TI})\right) \to \Omega ( \phi_0^{\star}(\mathscr{R}_B^*) ),
\end{align*}
given by
\begin{align*}
 & q(\eta)_x(a_1,...,a_p) = \eta_{(x,0)}(a_1,...,a_p),
 & a_1,...,a_p \in A_x.
\end{align*}
Write
\begin{align*}
 \Phi_{\epsilon} : \Omega( \phi^{\star}(\mathscr{R}_B^*) ) \to \Omega ( \phi_0^{\star}(\mathscr{R}_B^*) )
\end{align*}
for the composition of
\begin{align*}
 (\phi^{\mathcal{V}}_{\epsilon})^{\dagger} : \Gamma\left( \mathcal{A}_{TI} \right) \to \Gamma\left( i_0^*(\mathcal{A}_{TI})\right)
\end{align*}
with $q$, and note that it has degree zero, i.e., satisfies MRuth1). \\

\noindent \underline{\emph{Step 2. $\Phi_{\epsilon}$ satisfies MRuth2).}}  Let $\omega \in \Omega(A)$ and $\eta \in \Gamma\left( \mathcal{A}_{TI} \right)$. Consider the curves
\begin{align*}
 c_1,c_2:[0,1] \to \Omega( \phi_0^{\star}(\mathscr{R}_B^*) )
\end{align*}
given by
\begin{align*}
 & c_1(\epsilon) = \Phi_{\epsilon}(\omega \wedge \eta),
 & c_2(\epsilon) = \omega \wedge \Phi_{\epsilon}(\eta).
\end{align*}
For each $0 < \delta < 1$, the smooth curves
\begin{align*}
 \gamma_1,\gamma_2:[0,1-\delta] \to \Omega((\phi|_{A \times T[0,\delta]})^{\star}(\mathscr{R}_B^*))
\end{align*}
given by
\begin{align*}
 && \gamma_1(\epsilon) = (\phi^{\mathcal{V}}_{\epsilon})^{\dagger}(\omega \wedge \eta)
 && \text{and}
 && \gamma_2(\epsilon) = \omega \wedge (\phi^{\mathcal{V}}_{\epsilon})^{\dagger}(\eta),
\end{align*}
and where $\phi|_{A \times T[0,\delta]}$ denotes the restriction of $\phi$ to $A \times T[0,\delta] \subset A \times I$. Then note that $\gamma_1$ and $\gamma_2$ are respectively lifts of $c_1$ and $c_2$, in the sense that
\begin{align*}
 & c_1|_{[0,1-\delta]} = i_0^* \gamma_1,
 & c_2|_{[0,1-\delta]} = i_0^* \gamma_2,
\end{align*}
where
\begin{align*}
 i_0^* : \Omega((\phi|_{A \times T[0,\delta]})^{\star}(\mathscr{R}_B^*)) \to \Omega(\phi_0^{\star}(\mathscr{R}_B^*))
\end{align*}
is the map induced by the identification $i_0 : A \to A \times T\{0\}$. Because both $\gamma_1$ and $\gamma_2$ are solutions to the initial-value problem
\begin{align*}
 & \dot{\gamma} = \mathrm{L}\gamma,
 & \gamma(0) = \omega \wedge \eta,
\end{align*}
we deduce that $c_1|_{[0,1-\delta]} = c_2|_{[0,1-\delta]}$. Because $\delta < 1$ is arbitrary, it follows by continuity that
\begin{align*}
  & \Phi_{\epsilon}(\omega \wedge \eta) = \omega \wedge \Phi_{\epsilon}(\eta), & \epsilon \in I.
\end{align*}
\noindent \underline{\emph{Step 3. $\Phi_{\epsilon}$ satisfies MRuth3).}} 
Similarly to Step 2, one proves that the curves
\begin{align*}
 & c_1(\epsilon) = \Phi_{\epsilon}(\mathrm{D}_{\phi^{\star}(\mathscr{R}_B^*)}\eta),
 & c_2(\epsilon) = \mathrm{D}_{\phi_0^{\star}(\mathscr{R}_B^*)}\Phi_{\epsilon}(\eta)
\end{align*}
coincide by showing that the curves
\begin{align*}
 & \gamma_1(\epsilon) = (\phi^{\mathcal{V}}_{\epsilon})^{\dagger}(\mathrm{D}_{\phi^{\star}(\mathscr{R}_B^*)} \eta),
 & \gamma_2(\epsilon) = \mathrm{D}_{\phi^{\star}(\mathscr{R}_B^*)} (\phi^{\mathcal{V}}_{\epsilon})^{\dagger}(\eta)
\end{align*}
defined for all $\epsilon \in [0,1-\delta]$ satisfy the same initial-value problem. This implies that
\begin{align*}
 \Phi : \phi_0^{\star}(\mathscr{R}_B^*) \times I \to \phi^{\star}(\mathscr{R}_B^*)
\end{align*}
is a morphism of ruths.\\

\noindent \underline{\emph{Step 4. A homotopy of ruths.}} The linear map
\begin{align*}
 \widetilde{\Phi} : \Omega(\phi^{\star}(\mathscr{R}_B^*)) \to \Omega (\phi_0^{\star}(\mathscr{R}_B^*) \times TI )
\end{align*}
given by
\begin{align*}
 \widetilde{\Phi}(\eta) = \Phi(\eta) + \mathrm{d}\epsilon \wedge \Phi \left(i_{\tfrac{\partial}{\partial \epsilon}}\eta \right).
\end{align*}
defines a morphism of ruths
\begin{align*}
 \widetilde{\Phi} : \phi_0^{\star}(\mathscr{R}_B^*) \times TI \to \phi^{\star}(\mathscr{R}_B^*),
\end{align*}
since
\begin{align*}
 \mathrm{D}_{\phi_0^{\star}(\mathscr{R}_B^*) \times TI}\widetilde{\Phi}(\eta) = \mathrm{D}_{\phi_0^{\star}(\mathscr{R}_B^*)}\Phi(\eta) + \mathrm{d}\epsilon \wedge \tfrac{d}{d\epsilon}\Phi(\eta) - \mathrm{d}\epsilon \wedge \mathrm{D}_{\phi_0^{\star}(\mathscr{R}_B^*)}\Phi \left(i_{\tfrac{\partial}{\partial \epsilon}}\eta \right)
\end{align*}
and
\begin{align*}
 \widetilde{\Phi}(\mathrm{D}_{\phi^{\star}(\mathscr{R}_B^*)}\eta) = \Phi(\mathrm{D}_{\phi^{\star}(\mathscr{R}_B^*)}\eta) + \mathrm{d}\epsilon \wedge \Phi \left(i_{\tfrac{\partial}{\partial \epsilon}}\mathrm{D}_{\phi^{\star}(\mathscr{R}_B^*)}\eta \right)
\end{align*}
imply that
\begin{align*}
 \mathrm{D}_{\phi_0^{\star}(\mathscr{R}_B^*) \times TI}\widetilde{\Phi} - \widetilde{\Phi}\mathrm{D}_{\phi^{\star}(\mathscr{R}_B^*)}
\end{align*}
equals
\begin{align*}
 \left(\mathrm{D}_{\phi_0^{\star}(\mathscr{R}_B^*)}\Phi - \Phi \mathrm{D}_{\phi^{\star}(\mathscr{R}_B^*)}\right) + \mathrm{d}\epsilon \wedge \left(\tfrac{d}{d\epsilon}\Phi - \Phi \left(\mathrm{D}_{\phi^{\star}(\mathscr{R}_B^*)}i_{\tfrac{\partial}{\partial \epsilon}} + i_{\tfrac{\partial}{\partial \epsilon}}\mathrm{D}_{\phi^{\star}(\mathscr{R}_B^*)}\right)\right),
\end{align*}
and therefore
\begin{align*}
 \mathrm{D}_{\phi_0^{\star}(\mathscr{R}_B^*) \times TI}\widetilde{\Phi} = \widetilde{\Phi}\mathrm{D}_{\phi^{\star}(\mathscr{R}_B^*)},
\end{align*}
and so
\begin{align*}
 \widetilde{\Phi} : \phi_0^{\star}(\mathscr{R}_B^*) \times TI \diffto \phi^{\star}(\mathscr{R}_B^*)
\end{align*}
is a homotopy of ruths.
\end{proof}

For the pullback ruth, we have by the discussion above:

\vspace{0.2cm}
\begin{minipage}{0.9\textwidth}
 \begin{mdframed}[backgroundcolor=blue!5]
\begin{corollary}
Let $\mathscr{R}$ be a ruth on $B$, and let $\phi : A \times TI \to B$ be a morphism of Lie algebroids. Then there is an isomorphism of ruths
\begin{align*}
 \Phi : \phi_0^{\star}(\mathscr{R}) \diffto \phi_1^{\star}(\mathscr{R})
\end{align*}
with the property that the diagram below commutes:
\begin{align*}
 \xymatrix{
 & \mathrm{H}(\mathscr{R}^*) \ar[dl]_{\phi_1^*} \ar[dr]^{\phi_0^*} & \\
 \mathrm{H}(\phi_1^{\star}(\mathscr{R}^*)) \ar[rr]_{\Phi} && \mathrm{H}(\phi_0^{\star}(\mathscr{R}^*)).
 }
\end{align*}

\end{corollary}
\end{mdframed}
\end{minipage}

\section{The fibred picture}\label{sec : The fibred picture}

An {\bf extension} of Lie algebroids $\mathfrak{E} = (B \to \underline{B})$ \cite{Brahic08} is a surjective Lie algebroid map covering a surjective submersion $p_N:N \to \underline{N}$ :
\begin{align*}
 \xymatrix{
 B \ar@{=>}[d] \ar[r]^{P_B} & \underline{B} \ar@{=>}[d]\\
 N \ar[r]_{p_N} & \underline{N}
 }
\end{align*}
A {\bf morphism} $\mathfrak{E}_A \to \mathfrak{E}_B$ between extensions of Lie algebroids
\begin{align*}
 & \xymatrix{
 A \ar@{=>}[d] \ar[r]^{P_A} & \underline{A} \ar@{=>}[d]\\
 M \ar[r]_{p_M} & \underline{M}
 }
 & \xymatrix{
 B \ar@{=>}[d] \ar[r]^{P_B} & \underline{B} \ar@{=>}[d]\\
 N \ar[r]_{p_N} & \underline{N}
 }
\end{align*}
consists of Lie algebroid morphisms
\begin{align*}
 & \phi : A \to B, & \underline{\phi} : \underline{A} \to \underline{B}, 
\end{align*}
for which the diagram below commutes:
\begin{align*}
 \xymatrix{
 A \ar[d]_{P_A} \ar[r]^{\phi} & B \ar[d]^{P_B}\\
 \underline{A} \ar[r]_{\underline{\phi}} & \underline{B}
 }
\end{align*}
An {\bf Ehresmann connection} on $\mathfrak{E}$ is a $C^{\infty}(\underline{N})$-linear map
\begin{align*}
 & \mathrm{hor} : \Gamma(\underline{B}) \to \Gamma(B), 
 & P_B\mathrm{hor} = \mathrm{id}_{\underline{B}}.
\end{align*}
A such Ehresmann connection is {\bf complete} if $\mathrm{hor}(\underline{b})$ is a complete section of $B$ whenever $\underline{b}$ is a complete section of $\underline{B}$. An extension of Lie algebroids which admits a complete Ehresmann connection will be called a {\bf fibration} of Lie algebroids.

\begin{lemma}\label{lem : Pullback of an extension by a morphism}
 Let $\mathfrak{E} = (B \to \underline{B})$ be a extension of Lie algebroids, and let 
  \begin{align*}
  \xymatrix{
  A \ar[r]^{\phi} \ar@{=>}[d] & \underline{B} \ar@{=>}[d]\\
  M \ar[r]_{f} & \underline{N}
  }
 \end{align*}
 be a morphism of Lie algebroids. Then there is an induced {\bf pullback} extension $\phi^*(\mathfrak{E}) = (\phi^*(B) \to A)$, together with a pullback morphism of extensions of Lie algebroids
 \begin{align*}
  \phi_* : \phi^*(\mathfrak{E}) \to \mathfrak{E}.
 \end{align*}
\end{lemma}
\begin{proof}
 Consider the pullback submersion
 \begin{align*}
  \xymatrix{
  f^*(N) \ar[r]^{f_*} \ar[d]_{p_{f^*(N)}} & N \ar[d]^{p_N}\\
  M \ar[r]_{f} & \underline{N}
  }
 \end{align*}
 The fibred product $\phi^*(B) := A \times_{\underline{B}}B$ inherits the structure of a Lie subalgebroid $\phi^*(B) \Rightarrow f^*(N)$ of $A \times B \Rightarrow M \times N$ (see \cite[Proposition 7.14]{Eckhard}), and the canonical maps
 \begin{align*}
  & \phi_* : \phi^*(B) \to B,
  & \phi_*(a,b) = b,\\
  & P_{\phi^*(B)} : \phi^*(B) \to A,
  & P_{\phi^*(B)}(a,b)=a
 \end{align*}
 are Lie algebroid morphisms. Then
 \begin{align*}
  \xymatrix{
  \phi^*(B) \ar[r]^{P_{\phi^*(B)}} \ar@{=>}[d] & A \ar@{=>}[d]\\
  f^*(N) \ar[r]_{p_{f^*(N)}} & M
  }
 \end{align*}
 is an extension of Lie algebroids, and
 \begin{align*}
  \xymatrix{
  \phi^*(B) \ar[r]^{\phi_*} \ar[d]_{P_{\phi^*(B)}} & B \ar[d]^{P_B}\\
  A \ar[r]_{\phi} & \underline{B}
  }
 \end{align*}
 is a morphism of extensions of Lie algebroids.
\end{proof}

\begin{example}[Submersions by Lie algebroids]\normalfont
A \emph{submersion by Lie algebroids} \cite{Frejlich}
\begin{align*}
 \mathfrak{S}_B \ : \ B \Rightarrow N \stackrel{p_N}{\rmap} \underline{N},
\end{align*}
consists of a surjective submersion $p_N : N \to \underline{N}$ and a Lie algebroid $B \Rightarrow N$, with the property that each fibre of $p_N$ is transverse to $B$, and a morphism $\phi : \mathfrak{S}_A \to \mathfrak{S}_B$ from another submersion by Lie algebroids
\begin{align*}
 \mathfrak{S}_A \ : \ A \Rightarrow M \stackrel{p_M}{\rmap} \underline{M}
\end{align*}
is a Lie algebroid map $\phi : A \to B$ whose base map $f:M \to N$ is fibred:
\begin{align*}
 \xymatrix{
  M \ar[r]^f \ar[d]_{p_M}
& N \ar[d]^{p_N} \\
  \underline{M} \ar[r]_{\underline{f}}
& \underline{N}
}
\end{align*}
It is easy to see that there is an embedding
\begin{align*}
 && \xymatrix{
  A \ar[r]^{\phi} \ar@{=>}[d]
& B \ar@{=>}[d] \\
  M \ar[r]^f \ar[d]
& N \ar[d] \\
  \underline{M} \ar[r]_{\underline{f}}
& \underline{N}
 }
 && \mapsto
 && \xymatrix{
  A \ar[r]^{\phi} \ar[d]
& B \ar[d] \\
  T\underline{M} \ar[r]_{Tf}
& T\underline{N}
 }
\end{align*}
of the category of submersions by Lie algebroids into the category of extensions of Lie algebroids.
\end{example}

\begin{definition}\normalfont
 A {\bf fibration by ruths} $\mathscr{F}=(\mathfrak{E},\mathscr{R})$ is a fibration of Lie algebroids $\mathfrak{E} = (A \to \underline{A})$, together with a ruth $\mathscr{R} = (\mathrm{D}_{\mathscr{R}} : A \curvearrowright \mathcal{E})$.  A {\bf morphism}
 \begin{align*}
  \Phi : \mathcal{F}_A \to \mathcal{F}_B
 \end{align*}
between fibrations by ruths
 \begin{align*}
  && \mathcal{F}_A = (\mathfrak{E}_A, \mathscr{R}_A),
  && \mathfrak{E}_A = \left( A \to \underline{A} \right),
  && \mathscr{R}_A = (\mathrm{D}_{\mathscr{R}_A} : A \curvearrowright \mathcal{E}_A)\\
  && \mathcal{F}_B = (\mathfrak{E}_B, \mathscr{R}_B),
  && \mathfrak{E}_B = \left( B \to \underline{B} \right),
  && \mathscr{R}_B = (\mathrm{D}_{\mathscr{R}_B} : B \curvearrowright \mathcal{E}_B)
 \end{align*}
is a morphism of ruths
\begin{align*}
 \Phi : \mathscr{R}_A \to \mathscr{R}_B
\end{align*}
covering a morphism of extensions of Lie algebroids $\phi : \mathfrak{E}_A \to \mathfrak{E}_B$. 
\end{definition}

\begin{example}[Pullback of a fibration by ruths by a Lie algebroid morphism]\label{ex : Pullback of a fibration by ruths by a Lie algebroid morphism}\normalfont
 Let $\mathscr{F}=(\mathfrak{E},\mathscr{R})$ be the fibration by ruths
 \begin{align*}
  & \xymatrix{
 B \ar@{=>}[d] \ar[r]^{P_B} & \underline{B} \ar@{=>}[d]\\
 N \ar[r]_{p_N} & \underline{N}
 }
 & \mathscr{R} = (\mathrm{D}_{\mathscr{R}} : B \curvearrowright \mathcal{E})
 \end{align*}
and let
 \begin{align*}
  \xymatrix{
  A \ar@{=>}[d] \ar[r]^{\phi} & \underline{B} \ar@{=>}[d]\\
  M \ar[r]_{f} & \underline{N}
  }
 \end{align*}
 be a Lie algebroid morphism. Then
\begin{align*}
 \phi^*(B):= A \times_{\underline{B}} B \Rightarrow f^*(N) = M \times_{\underline{N}}N
\end{align*}
is a Lie algebroid \cite[Proposition 7.14]{Eckhard},
\begin{align*}
  \xymatrix{
 \phi^*(B) \ar@{=>}[d] \ar[r]^{P_{\phi^*(B)}} & A \ar@{=>}[d]\\
 f^*(N) \ar[r]_{p_{f^*(N)}} & M
 }
 \end{align*}
 is a Lie algebroid fibration, where $P_A(a,b)=a$, $p_M(x,y)=x$ and
\begin{align*}
  \xymatrix{
 \phi^*(B) \ar[d] \ar[r]^{\phi_*} & B \ar[d]\\
 A \ar[r]_{\phi} & \underline{B}
 }
 \end{align*}
 is a morphism of Lie algebroid fibrations, where $\phi_*(a,b)=b$. There is then an induced fibration by ruths $(\phi_*)^{\star}(\mathscr{F})$
 \begin{align*}
  & \xymatrix{
 \phi^*(B) \ar@{=>}[d] \ar[r]^{P_B} & A \ar@{=>}[d]\\
 f^*(N) \ar[r]_{\phantom{12}p_{f^*(N)}} & M
 }
 & (\phi_*)^{\star}(\mathscr{R}) = (\mathrm{D}_{(\phi_*)^{\star}(\mathscr{R})} : \phi^*(B) \curvearrowright f^*(N)\times_N\mathcal{E})
 \end{align*}
and a morphism of fibrations by ruths
\begin{align*}
 (\phi_*)_{\star} : (\phi_*)^{\star}(\mathscr{R}) \to \mathscr{R}.
\end{align*}
\end{example}

\vspace{0.2cm}
\begin{minipage}{0.9\textwidth}
 \begin{mdframed}[backgroundcolor=blue!5]
\begin{proposition}\label{pro : local triviality of fibration by ruths under lie homotopy}
Let $\mathcal{F}=(B \to \underline{B},\mathrm{D}_{\mathscr{R}}:B \curvearrowright \mathcal{E})$ be a fibration by ruths, and let
\begin{align*}
 \phi : A \times TI \to \underline{B}
\end{align*}
be a Lie algebroid morphism. Then the fibrations by ruths
\begin{align*}
 && (\phi_*)^{\star}(\mathcal{F}) && \text{and} && ({\phi_0}_*)^{\star}(\mathcal{F}) \times TI
\end{align*}
are isomorphic.
\end{proposition}
\end{mdframed}
\end{minipage}
\begin{proof}
 Consider the pullback diagram of Lie algebroids
 \begin{align*}
  \xymatrix{
    \phi^*(B) \ar[r]^{\phi_*} \ar[d]_{P_{\phi^*(B)}}
  & B \ar[d]^{P_B} \\
    A \times TI \ar[r]_{\phi}
  & \underline{B}
  }
 \end{align*}
The complete connection $\mathrm{hor}_{\underline{B}} : p_N^*(\underline{B}) \to B$ induces a complete connection
\begin{align*}
 & \mathrm{hor}_{A \times TI} : p_{f^*(N)}(A \times TI) \to \phi^*(B),
 & \mathrm{hor}_{A \times TI}(a):=(a,\mathrm{hor}_{\underline{B}}(\phi(a)))
\end{align*}
and the local flow of $\mathrm{hor}_{A \times TI}(\tfrac{\partial}{\partial \epsilon})$ defines an isomorphism of fibrations
\begin{align*}
\xymatrix{
  \phi_0^*(B) \times TI \ar[r]^{\psi}_{\simeq} \ar[d]_{(P_{\phi_0^*(B)},\mathrm{id})} 
& \phi^*(B) \ar[d]^{P_{\phi^*(B)}} \\
  A \times TI \ar@{=}[r]
& A \times TI
}
\end{align*}
which restricts to the inclusion $\phi_0^*(B) \to \phi^*(B)$ at $\epsilon = 0$. Consider the pullback ruth
\begin{align*}
 (\phi_*)^{\star}(\mathscr{R}) \ : \ \phi^*(B) \curvearrowright f^*(N) \times_N \mathcal{E}.
\end{align*}
Then
\begin{align*}
 ({\phi_0}_*)^{\star}(\mathscr{R}) = \psi_0^{\star}(\phi_*)^{\star}(\mathscr{R})
\end{align*}
and by Theorem \ref{thm : local triviality ruths}, there is an isomorphism of ruths
\begin{align*}
 \Phi' : ({\phi_0}_*)^{\star}(\mathscr{R}) \times TI \diffto \psi^{\star}(\phi_*)^{\star}(\mathscr{R}).
\end{align*}
over $\mathrm{id} : A \times TI \to A \times TI$. Therefore
\begin{align*}
 \Phi = \psi_{\star} \circ \Phi' : ({\phi_0}_*)^{\star}(\mathscr{R}) \times TI \diffto (\phi_*)^{\star}(\mathscr{R})
\end{align*}
is an isomorphism of fibrations by ruths
\begin{align*}
 \Phi : ({\phi_0}_*)^{\star}(\mathcal{F}) \times TI \diffto (\phi_*)^{\star}(\mathcal{F}).
\end{align*}
\end{proof}

A fibration by ruths $\mathcal{F}=(\mathfrak{E},\mathscr{R})$,
\begin{align*}
 & \mathfrak{E} \ = \ 
 \xymatrix{
  A \ar[r]^{P_A} \ar[d]
& \underline{A} \ar[d] \\
  M \ar[r]_{p_M}
& \underline{M}
 }
 & \mathscr{R} = A \curvearrowright \mathcal{E}
\end{align*}
is {\bf locally trivial} if every point $x \in \underline{M}$ has an open neighborhood $U \subset \underline{M}$ for which there is an isomorphism of fibrations by ruths
\begin{align*}
 (({i_x}_!)_*)^{\star}(\mathscr{R}) \times \underline{A}|_U \diffto (({i_U}_!)_*)^{\star}.(\mathscr{R}).
\end{align*}

Generalizing \cite[Theorem 3]{Frejlich}, we have:

\vspace{0.2cm}
\begin{minipage}{0.9\textwidth}
 \begin{mdframed}[backgroundcolor=blue!5]
\begin{corollary}
A submersion by ruths $\mathcal{F}=(\mathfrak{E},\mathscr{R})$ is locally trivial exactly when the underlying submersion by Lie algebroids $\mathfrak{E}$ is locally trivial.
\end{corollary}
\end{mdframed}
\end{minipage}
\begin{proof}
 Let $\mathfrak{E} = (A \to T\underline{M})$ be a submersion by Lie algebroids, and let $\mathscr{R} : A \curvearrowright \mathcal{E}$ be a ruth. Then every $x \in \underline{M}$ has a \emph{contractible} open neighborhood $U \subset \underline{M}$, that is, one for which there is a smooth homotopy
 \begin{align*}
  && h : U \times I \to U,
  && h(y,0)=x,
  && h(y,1)=y.
 \end{align*}
Observe that
\begin{align*}
 & ({{h_0}_!}_*)^{\star}(\mathscr{R}) = ({{i_x}_!}_*)^{\star}(\mathscr{R}) \times TU,
 & ({{h_1}_!}_*)^{\star}(\mathscr{R}) = ({{i_U}_!}_*)^{\star}(\mathscr{R}).
\end{align*}
If $\mathfrak{E}$ has a complete Ehresmann connection, then by Proposition \ref{pro : local triviality of fibration by ruths under lie homotopy}, there is an isomorphism of submersions by ruths \begin{align*}
 ({{i_x}_!}_*)^{\star}(\mathscr{R}) \times TU \diffto ({{i_U}_!}_*)^{\star}(\mathscr{R}).
\end{align*}
\end{proof}

\section{Remarks}\label{sec : Remarks}

\addtocounter{litreview}{1}

\noindent {\textbf {\Roman{litreview})}} \label{lit : morphism} One of the novelties in the paper concerns the definition of ``morphism of ruths'' in the case where the underlying Lie algebroids are defined over different manifolds, Definition \ref{def : Morphism of ruths}. In the foundational paper \cite{Abad_Crainic}, a \emph{morphism} $\mathscr{R}_E \to \mathscr{R}_F$ between ruths
\begin{align*}
 & \mathscr{R}_{\mathcal{E}} \ : \ A \overset{\mathrm{D}_{\mathscr{R}_{\mathcal{E}}}}{\curvearrowright} \mathcal{E},
 & \mathscr{R}_{\mathcal{F}} \ : \ A \overset{\mathrm{D}_{\mathscr{R}_{\mathcal{F}}}}{\curvearrowright} \mathcal{F}
\end{align*}
over the same Lie algebroid $A \Rightarrow M$ is a defined as a map of DGA modules
\begin{align}\label{eq : covariant morphism, same base}
 \Psi : \Omega(\mathscr{R}_{\mathcal{E}}) \to \Omega(\mathscr{R}_{\mathcal{F}}),
\end{align}
that is, a linear map of degree zero for which
\begin{align*}
 & \Psi(\omega \wedge \eta) = \omega \wedge \Psi(\eta),
 & \Psi(\mathrm{D}_{\mathscr{R}_{\mathcal{E}}}\eta) = \mathrm{D}_{\mathscr{R}_{\mathcal{F}}}\Psi(\eta)
\end{align*}
for all $\omega \in \Omega(A)$ and $\eta \in \Omega(\mathscr{R}_{\mathcal{E}})$. Observe that there is a bijective correspondence between $\Omega(A)$-module maps as in \eqref{eq : covariant morphism, same base} and $\Omega(A)$-module maps
\begin{align*}
 \Psi^{\vee} : \Omega(\mathscr{R}_{\mathcal{F}}^*) \to \Omega(\mathscr{R}_{\mathcal{E}}^*),
\end{align*}
determined by the condition that
\begin{align*}
 \langle \Psi(\eta),\tau \rangle = \langle \eta,\Psi^{\vee}(\tau) \rangle
\end{align*}
for all $\eta \in \Omega(\mathscr{R}_{\mathcal{E}})$ and $\tau \in \Omega(\mathscr{R}_{\mathcal{F}}^*)$. Because
\begin{align*}
 \langle \eta,\mathrm{D}_{\mathscr{R}_{\mathcal{E}}^*}\Psi^{\vee}(\tau)\rangle - \langle \eta,\Psi^{\vee}\mathrm{D}_{\mathscr{R}_{\mathcal{F}}^*}(\tau)\rangle & = \langle \mathrm{D}_{\mathscr{R}_{\mathcal{F}}}\Psi(\eta),\tau\rangle - \langle \Psi\mathrm{D}_{\mathscr{R}_{\mathcal{E}}}(\eta),\tau\rangle,
\end{align*}
$\Psi$ is a cochain map exactly when $\Psi^{\vee}$ is a cochain map. Therefore this ``covariant'' perspective selects the same morphisms as our ``contravariant'' proposal for ``morphism of ruths'' (Definition \ref{def : Morphism of ruths}). As in Remark \ref{rem : dual coefficients}, the contravariant definition is unavoidable if one is to consider morphisms betweens ruths defined on different manifolds.\\

\noindent \lit A second novelty in the manuscript concerns the notion in subsection \ref{subsec : Pullbacks of ruths by a Lie algebroid morphism} of ``pullback'' 
\begin{align*}
\phi^{\star}(\mathscr{R}) \ : \ A \overset{\mathrm{D}_{\phi^{\star}(\mathscr{R})}}{\curvearrowright} \underline{\phi}^*(\mathcal{E} )
\end{align*}
of a ruth
\begin{align*}
\mathscr{R} \ : \ B \overset{\mathrm{D}_{\mathscr{R}}}{\curvearrowright} \mathcal{E} 
\end{align*}
under an arbitrary Lie algebroid map
\begin{align*}
 \xymatrix{
 A \ar@{=>}[d] \ar[r]^{\phi} & B \ar@{=>}[d] \\
 M \ar[r]_{f} & N
 }
\end{align*}
The case of the pullback $(f_!)^{\star}(\mathscr{R})$ of a ruth by a smooth map $f:M \to N$ transverse to the Lie algebroid $B$ appeared in \cite[Example 4.3]{Abad_Crainic}. In \cite[Section 2.2]{Drummond_Jotz_Ortiz}, pullbacks are considered in the case where $\phi : A \to B$ is a Lie algebroid map covering the identity $f=\mathrm{id}$:
\begin{align*}
 \xymatrix{
 A \ar@{=>}[d] \ar[r]^{\phi} & B \ar@{=>}[d] \\
 M \ar@{=}[r] & M
 }
\end{align*}

\noindent \lit In the spirit of \cite{Abad_Crainic}, a ``covariant'' version can be given of our proposal of morphism $\mathscr{R}_A \to \mathscr{R}_B$ between ruths on different bases, in light of Lemma \ref{lem : morphisms factor through pullback}: it consists of a Lie algebroid morphism $\phi : A \to B$, together with a homomorphism of modules of DGAs:
\begin{align}\label{eq : covariant morphism}
 \Psi : \Omega(\mathscr{R}_A) \to \Omega(\phi^{\star}(\mathscr{R}_B))
\end{align}
over $\mathrm{id} : A \to A$:
\begin{align*}
 & \Psi(\omega \wedge \eta) = \omega \wedge \Psi(\eta),
 & \Psi(\mathrm{D}_{\mathscr{R}_A} \eta) = \mathrm{D}_{\phi^{\star}(\mathscr{R}_B)} \Psi(\eta)
\end{align*}
for all $\omega \in \Omega(A)$ and $\eta \in \Omega(\mathscr{R}_A)$. To a such homomorphism \eqref{eq : covariant morphism} there corresponds the composition $\Phi$ of
\begin{align*}
 \Omega(\mathscr{R}_B^*) \stackrel{\phi^*}{\rmap} \Omega(\phi^{\star}(\mathscr{R}_B^*)) \stackrel{\Psi^{\vee}}{\rmap} \Omega(\mathscr{R}_A^*),
\end{align*}
which gives a morphism of ruths
\begin{align}\label{eq : contravariant morphism}
 \Phi : \mathscr{R}_A \to \mathscr{R}_B
\end{align}
in our ``contravariant'' proposal. We find the ``contravariant'' approach more desirable, in that a morphism of ruths \eqref{eq : contravariant morphism} gives a map
\begin{align*}
 \Phi : \Omega(\mathscr{R}_B^*) \to \Omega(\mathscr{R}_A^*),
\end{align*}
whereas no map $\Omega(\mathscr{R}_A) \to \Omega(\mathscr{R}_B)$ is implied. Rather, there is an implied relation $\mathrm{Rel} \subset \Omega(\mathscr{R}_A) \times \Omega(\mathscr{R}_B)$, exactly as in Remark \ref{rem : dual coefficients}.\\

\noindent \lit That complete homotopies give rise to Lie algebroid isomorphisms appears in \cite[Proposition 1.1.15]{Pastina_Vitagliano}, where one also finds a proof of homotopy invariance of deformation cohomology, which is a special case of Theorem \ref{thm : cohomology} (see \cite[Theorem 3.11]{Abad_Crainic}).\\

\noindent \lit Theorem \ref{thm : cohomology} has a large intersection with the homotopy invariance of morphisms of $L_{\infty}$-algebroids of \cite{Gengoux_Lavau_Strobl,Caseiro_Gengoux} --- at least in the case where the homotopies are between morphisms covering the identity --- via the semidirect construction of \cite[Proposition 13]{Caseiro_Gengoux}.

\end{document}